\documentclass[preprint,12pt]{elsarticle}
\usepackage{amsthm,amsmath,amssymb}
\usepackage{graphicx}
\usepackage[colorlinks=true,citecolor=black,linkcolor=black,urlcolor=blue]{hyperref}
\usepackage{mathrsfs}

\newtheorem{thm}{Theorem}[section]
\newtheorem{Lemma}[thm]{Lemma}

\newtheorem{defn}[thm]{Definition}

\journal{}

\begin{document}

\title{The Minimal Coloring Number Of Any Non-splittable $\mathbb{Z}$-colorable Link Is Four}
\author{Meiqiao Zhang, Xian'an Jin, Qingying Deng\\
\small School of Mathematical Sciences\\[-0.8ex]
\small Xiamen University\\[-0.8ex]
\small P. R. China\\
\small\tt Email:xajin@xmu.edu.cn
}
\begin{abstract}
K. Ichihara and E. Matsudo introduced the notions of $\mathbb{Z}$-colorable links and the minimal coloring number for $\mathbb{Z}$-colorable links, which is one of invariants for links. They proved that the lower bound of minimal coloring number of a non-splittable $\mathbb{Z}$-colorable link is 4. In this paper, we show the minimal coloring number of any non-splittable $\mathbb{Z}$-colorable link is exactly 4.
\end{abstract}
\maketitle

\section{Introduction}

Imitating Fox-coloring \cite{f1} and the minimal coloring number \cite{h1} for links with Fox colorings, in order to deal with links of determinant 0, in \cite{1}, K. Ichihara and E. Matsudo introduced the notions of $\mathbb{Z}$-coloring and minimal coloring number, denoted by $mincol_{\mathbb{Z}}(L)$, for $\mathbb{Z}$-colorable links, which is one of invariants for links.

\begin{defn}
Let $L$ be a link and $D$ a diagram of $L$. We consider a map $\gamma$ :  \{arcs of $D$\}$ \xrightarrow{}  \mathbb{Z} $. If $\gamma$ satisfies the condition $2\gamma(a)=\gamma(b)+\gamma(c)$ at each crossing of $D$ with the over arc $a$ and the under arcs $b$ and $c$, then $\gamma$ is called a\emph{ $\mathbb{Z}$-coloring} on $D$. A $\mathbb{Z}$-coloring which assigns the same color to all arcs of the diagram is called the \emph{trivial $\mathbb{Z}$-coloring}. A link is called \emph{$\mathbb{Z}$-colorable} if it has a diagram admitting a non-trivial $\mathbb{Z}$-coloring.
\end{defn}

\begin{defn}
Let us consider the cardinality of the image of a non-trivial $\mathbb{Z}$-coloring on a diagram of a $\mathbb{Z}$-colorable link $L$. We call the minimum of such cardinalities among all non-trivial $\mathbb{Z}$-colorings on all diagrams of $L$ the \emph{minimal coloring number} of $L$, and denote it by \emph{$mincol_{\mathbb{Z}}(L)$}.
\end{defn}

\begin{defn}
Let $L$ be a $\mathbb{Z}$-colorable link, and $\gamma$ a non-trivial $\mathbb{Z}$-coloring on a diagram $D$ of $L$. Suppose that there exists a positive integer $d$ such that, at all the crossings in $D$, the differences between the colors of the over arcs and the under arcs are $d$ or $0$. Then we call $\gamma$ a \emph{simple $\mathbb{Z}$-coloring}.
\end{defn}

Then they proved:

\begin{thm}\label{bef}
\begin{enumerate}
\item Let $L$ be a non-splittable $\mathbb{Z}$-colorable link. Then $mincol_{\mathbb{Z}}(L) \geq 4$.
\item Let $L$ be a non-splittable $\mathbb{Z}$-colorable link. If there exists a simple $\mathbb{Z}$-coloring on a diagram of L, then $mincol_{\mathbb{Z}}(L)$ = 4.
\item If a non-splittable link $L$ admits a $\mathbb{Z}$-coloring with five colors, then $mincol_{\mathbb{Z}}(L)$ = 4.
\end{enumerate}
\end{thm}

In the end of \cite{1}, they posed two questions:
\vspace{2ex}
\newline
\noindent{\bf Question 1.5.}
\begin{enumerate}
\item Does $mincol_{\mathbb{Z}}(L) = 4$ always hold for any non-splittable $\mathbb{Z}$-colorable link $L$?
\item Does every non-splittable $\mathbb{Z}$-colorable link admit a simple $\mathbb{Z}$-coloring?
\end{enumerate}

In this paper, we give a positive answer to Question 1.5 (2), and hence a positive answer to Question 1.5 (1) by Theorem \ref{bef} (2).

\section{Main result and its proof}

For convenience, we denote the crossing with over arc $b$ and under arcs $a, c$ by $a|b|c$. Let $d$ be a nonnegative integer, we call a crossing a \emph{d-diff} one if the difference between the colors of the over arc and the under arcs is $d$. In the figures through the rest of this paper, we don't distinguish the over arc and under arcs of the uni-colored crossings. Moreover, a uni-colored $b|b|b$ crossing often represents a finite number, including 0, of $b|b|b$ crossings.

\begin{defn}
Let $L$ be a $\mathbb{Z}$-colorable link, and $\gamma$ a non-trivial $\mathbb{Z}$-coloring on a diagram $D$ of $L$. Two non-0-diff crossings are \emph{adjacent} if there are only a finite number of 0-diff crossings between them, as shown in the Figure 1.1 (containing 4 cases and 10 subcases altogether).
\end{defn}

In this section we shall prove:

\begin{thm}\label{main}
Any non-splittable $\mathbb{Z}$-colorable link admits a simple $\mathbb{Z}$-coloring.
\end{thm}

To prove Theorem \ref{main}, we need two lemmas.

\begin{figure}[htbp]
\centering
\includegraphics[height=5.5cm]{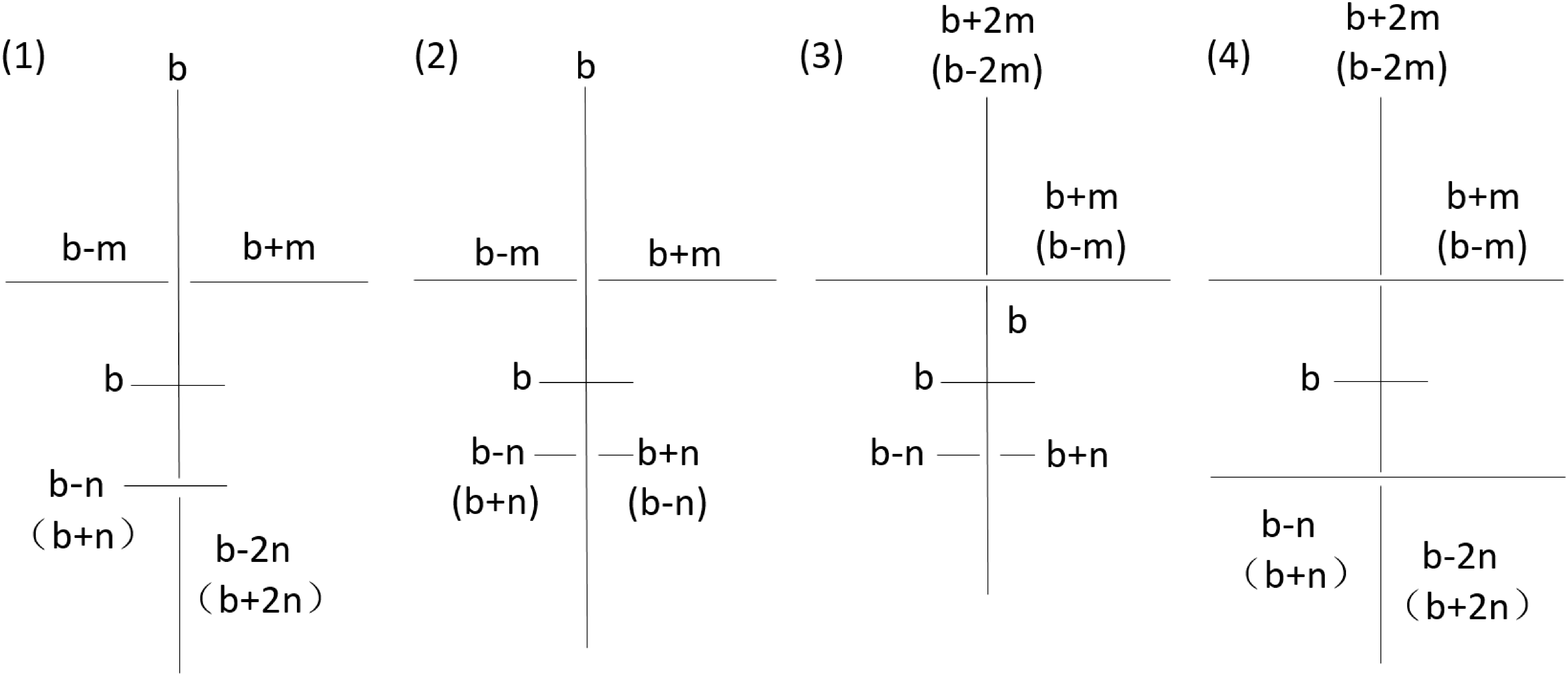}
\centerline{Figure 1.1}
\end{figure}

\begin{Lemma}\label{11}
Let $L$ be a $ \mathbb{Z} $-colorable link, and $\gamma$ a $ \mathbb{Z} $-coloring on a diagram $D$ of $L$. If there exists a pair of adjacent $n$-diff crossing and $qn$-diff crossing ($q \geq 2$, $q\in \mathbb{N}^{+}$), then the $qn$-diff crossing can be eliminated by equivalent local moves. Moreover, any newly created crossing in the process of elimination is either a $0$-diff crossing or an $n$-diff crossing.
\end{Lemma}

\begin{proof} We prove it case by case. Take $m=qn$ in Figure 1.1.

\noindent{\bf Case 1}: See Figure 1.1 (1). In this case we prove Lemma \ref{11} by induction on $q$. When $q=2$, we can eliminate $b-2n|b|b+2n$ as shown in Figure 1.1.1, and any newly created crossing in the process of elimination is either a $0-$diff crossing or an $n-$diff crossing.
\begin{figure}[htbp]
\centering
\includegraphics[height=6cm]{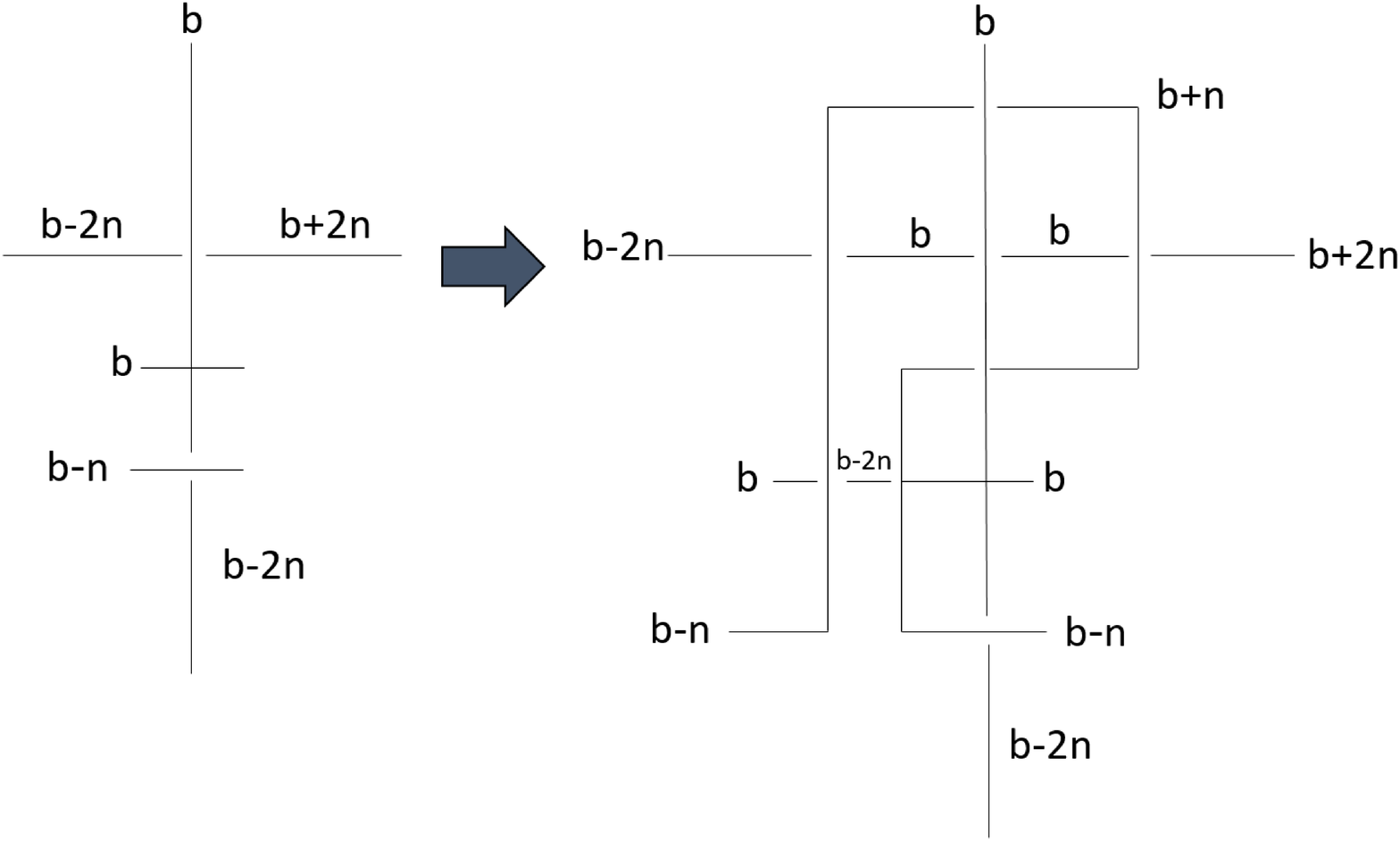}
\includegraphics[height=6cm]{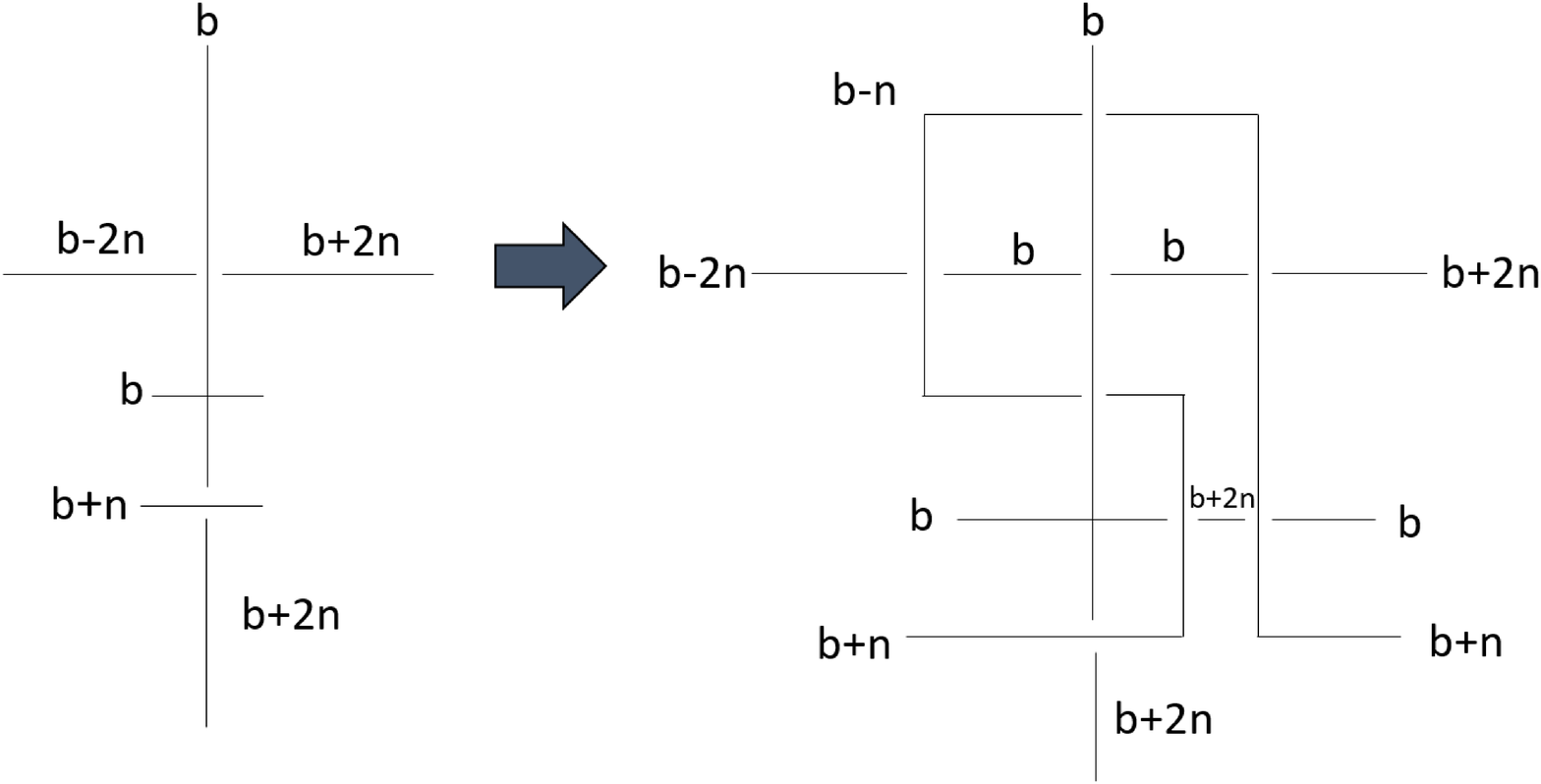}
\centerline{Figure 1.1.1}
\end{figure}
When $q=3$, we can eliminate $b-3n|b|b+3n$ as shown in Figure 1.1.2, and any newly created crossing in the process of elimination is either a $0-$diff crossing or an $n-$diff crossing.
\begin{figure}[htbp]
\centering
\includegraphics[height=6cm]{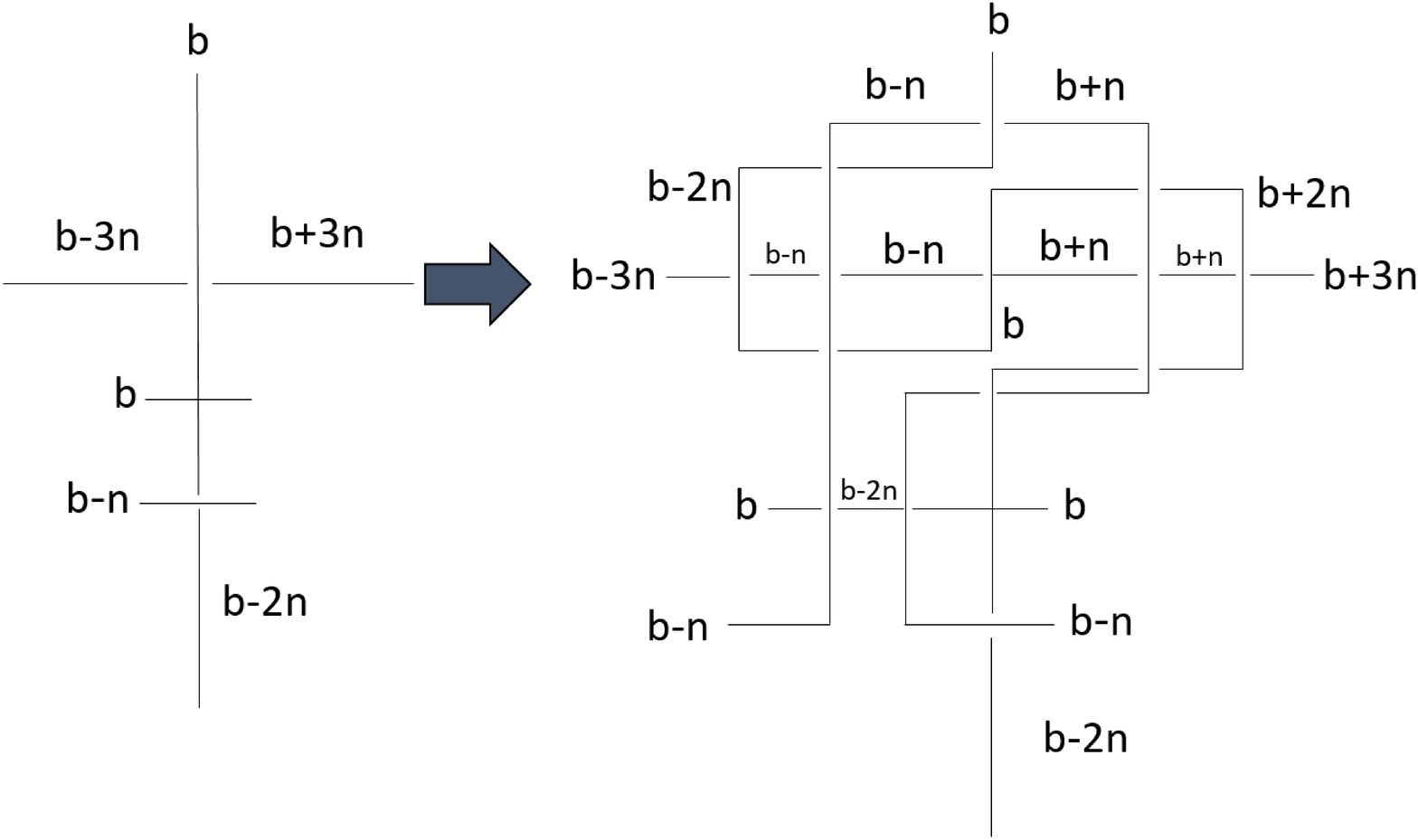}
\includegraphics[height=6cm]{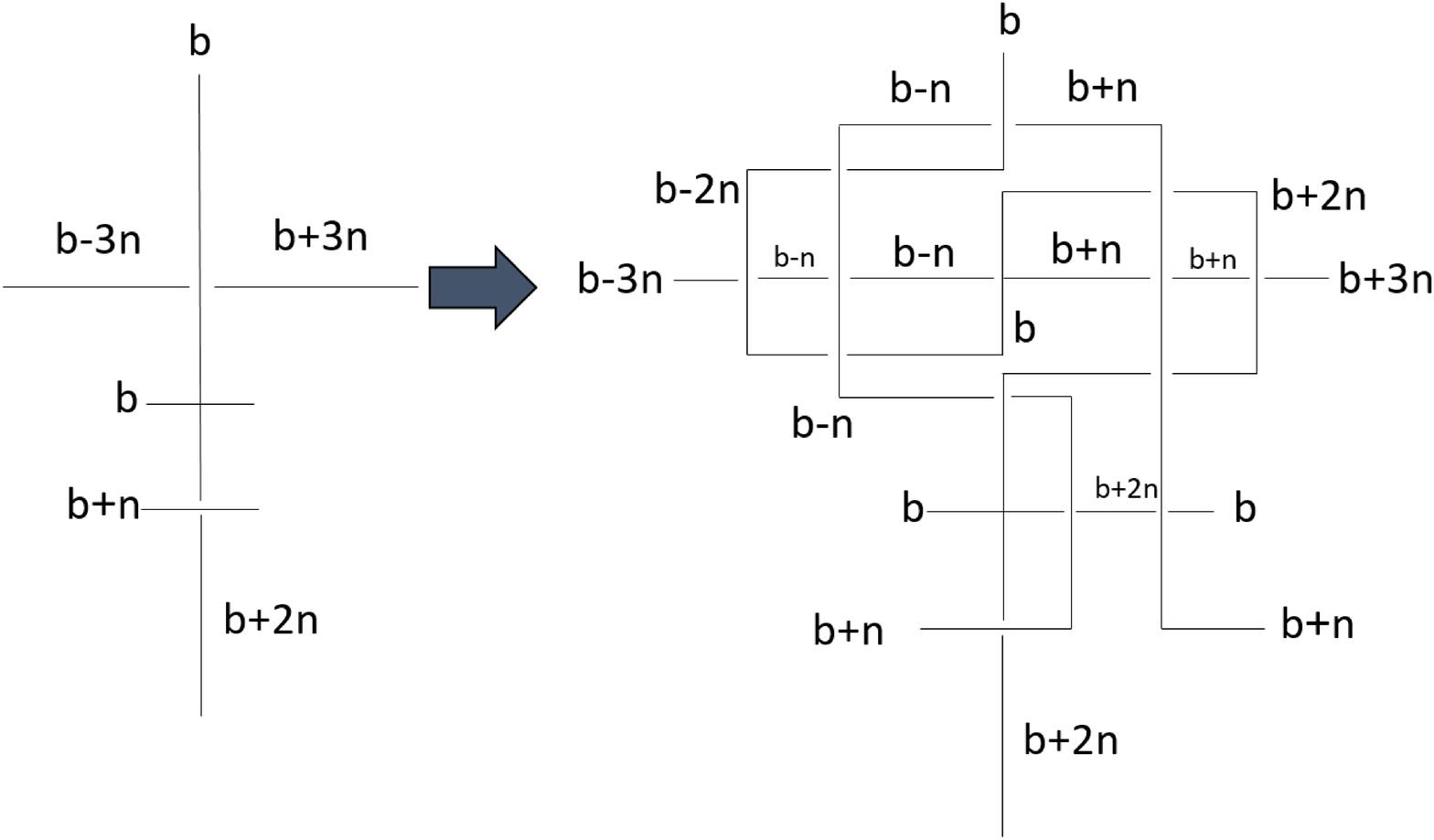}
\centerline{Figure 1.1.2}
\end{figure}
Now we assume that when $q\leq k-1$, $k\geq4$, the lemma holds, and shall prove when $q=k$, the lemma also holds. When $q=k$, we can reduce it to $q=k-2$ as shown in Figure 1.1.3.

\begin{figure}[htbp]
\centering
\includegraphics[height=6.2cm]{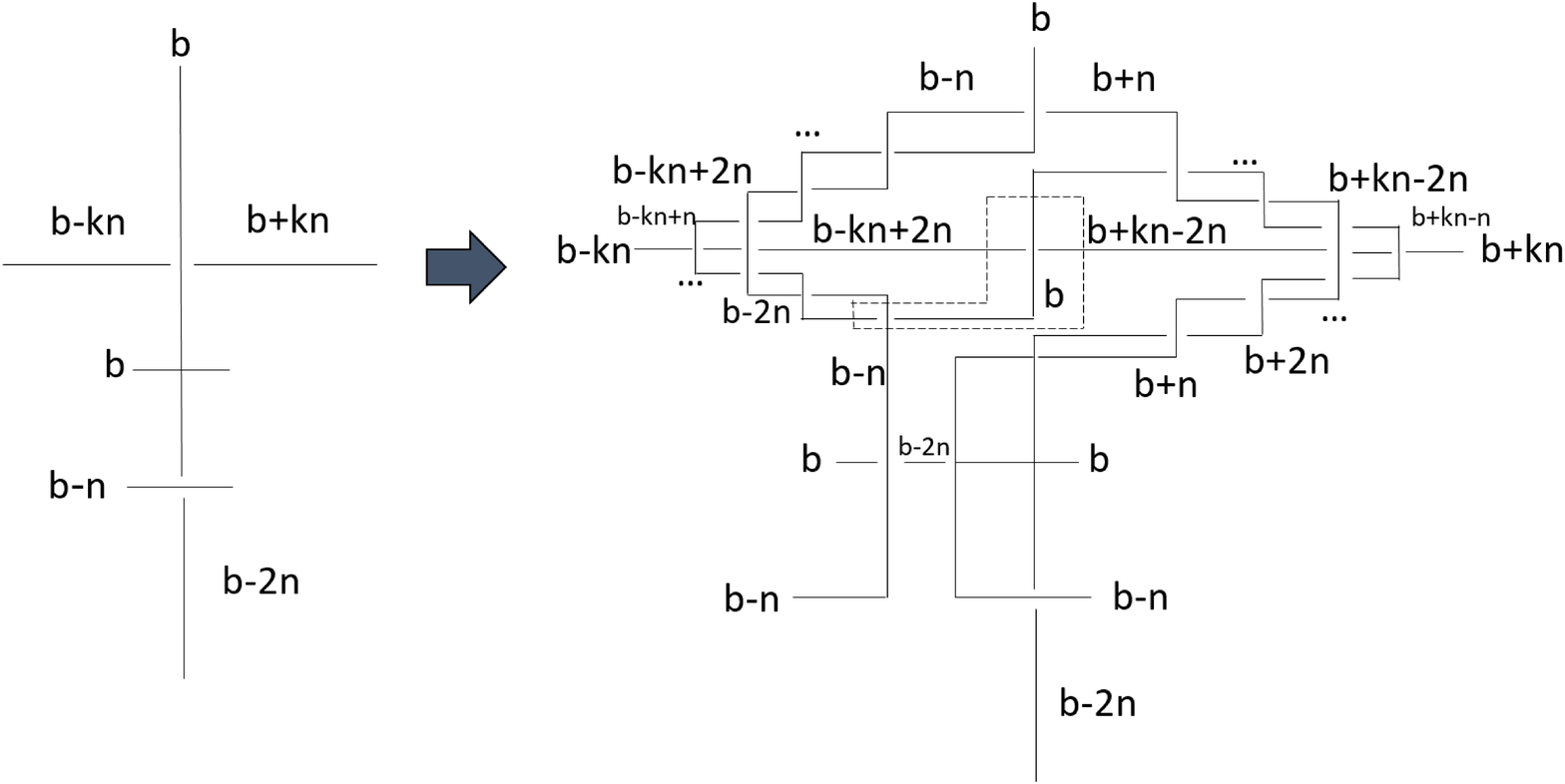}
\includegraphics[height=5.8cm]{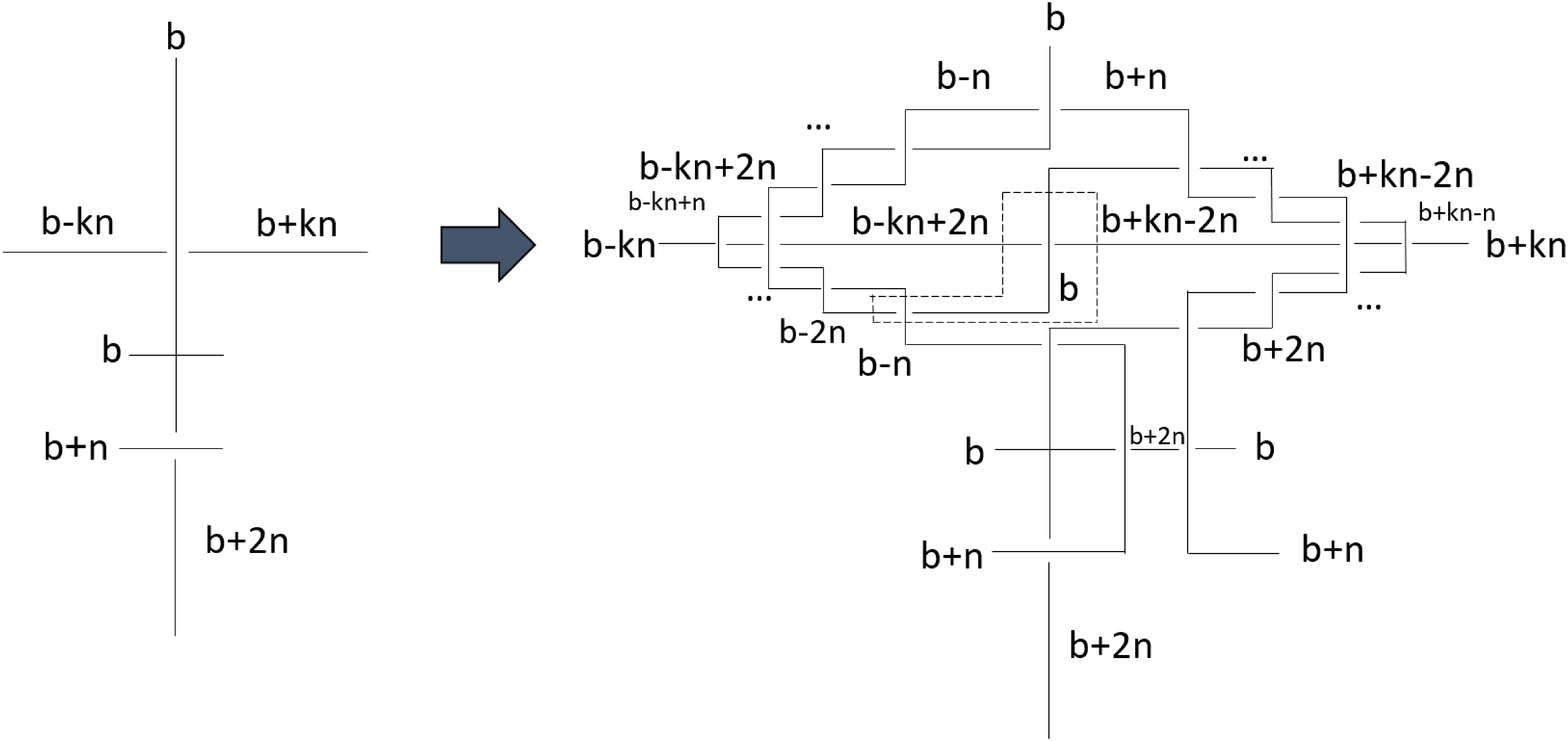}
\quote{Figure 1.1.3: The structure with newly created $b-(k-2)n|b|b+(k-2)n$ contained in the dashed box is the Case 1 with $q=k-2$. We can eliminate it by induction.}
\end{figure}

\noindent{\bf Case 2}: See Figure 1.1 (2). When $q=2$, we can eliminate $b-2n|b|b+2n$ as shown in Figure 1.2.1, and any newly created crossing in the process of elimination is either a $0-$diff crossing or an $n-$diff crossing.
\begin{figure}[htbp]
\centering
\includegraphics[height=5cm]{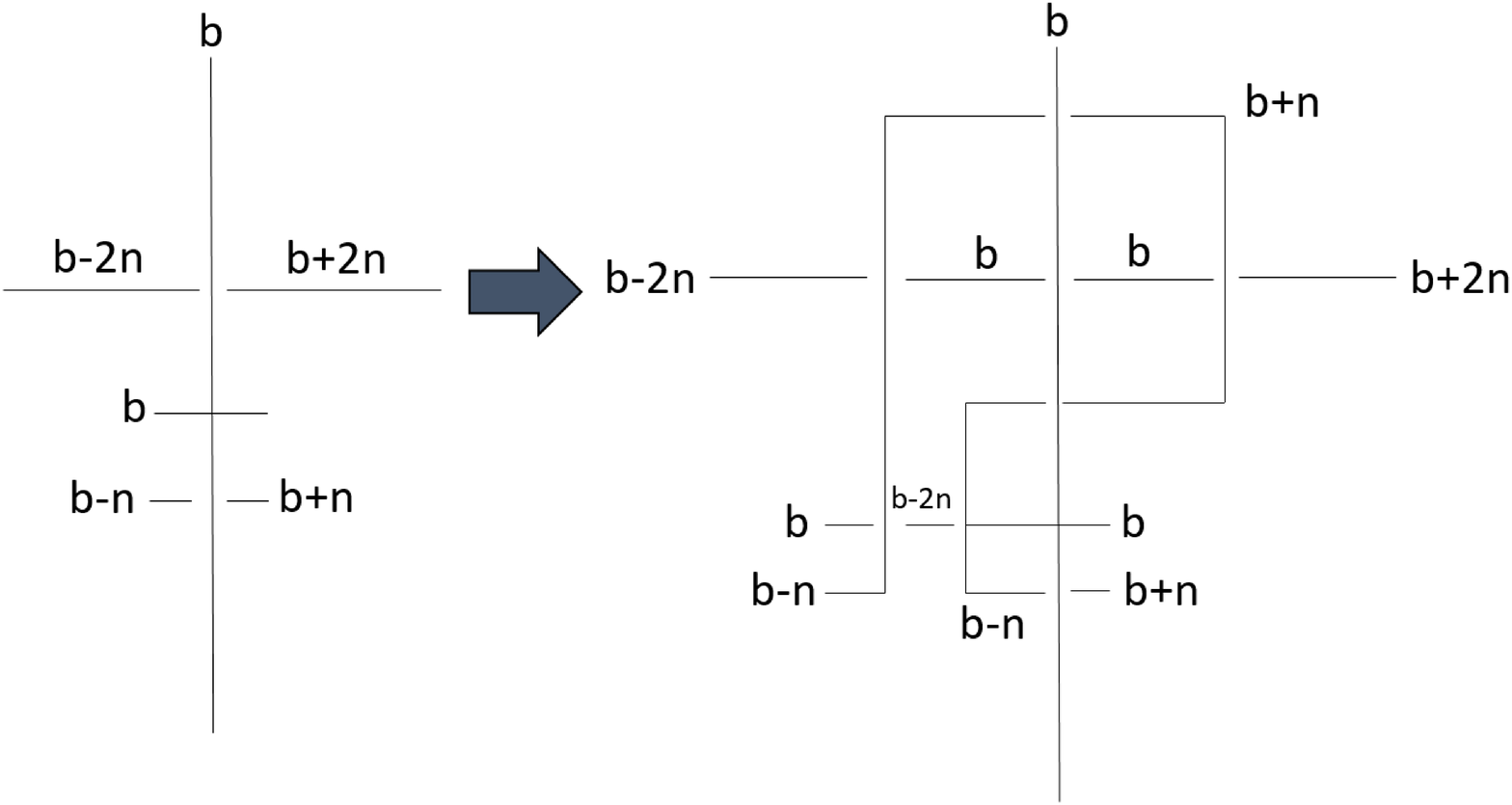}
\includegraphics[height=5cm]{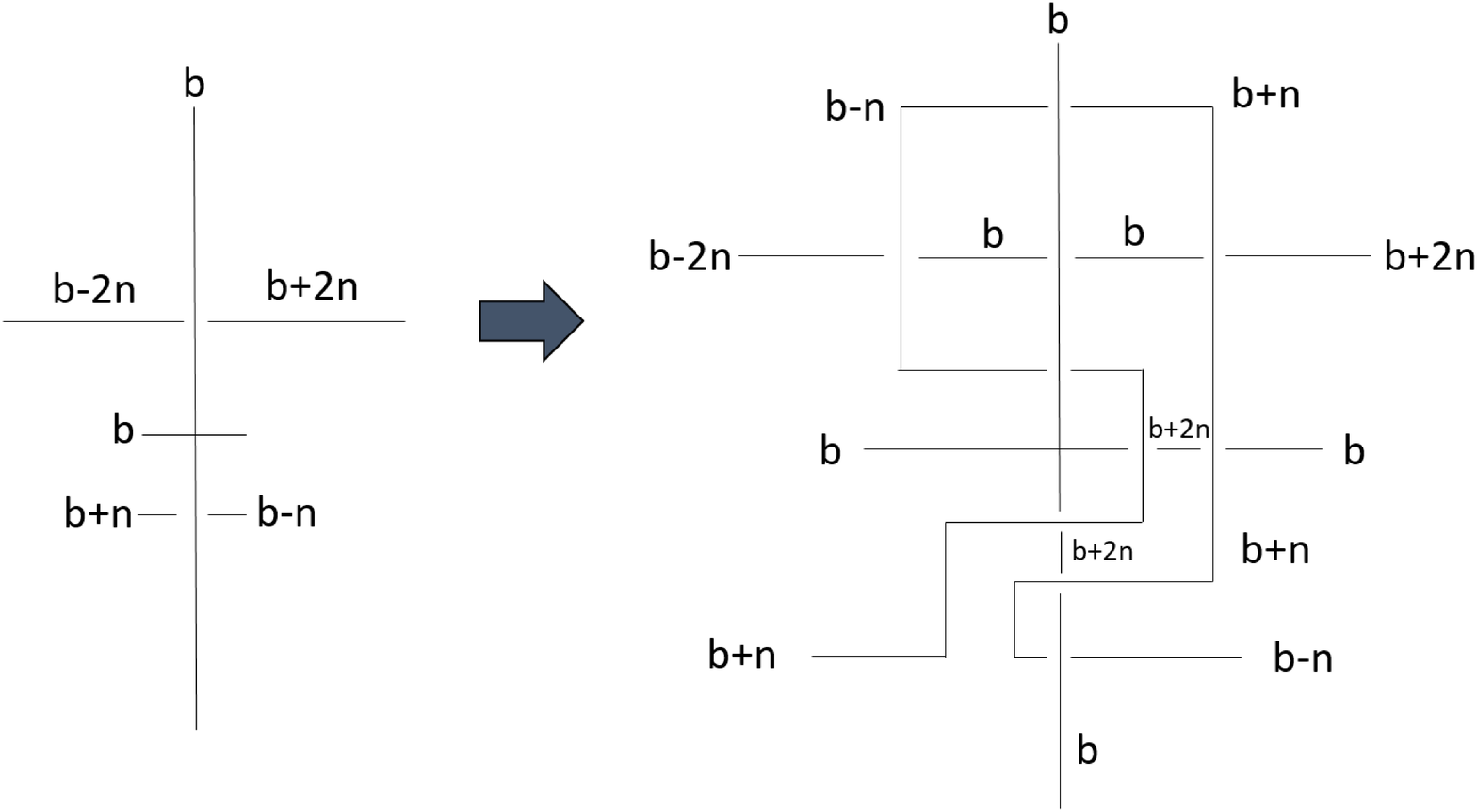}
\centerline{Figure 1.2.1}
\end{figure}
When $q=k$, $k\geq3$, we can operate as shown in Figure 1.2.2. Although there is newly created $b-(k-2)n|b|b+(k-2)n$ crossing, we can use the structure contained in the dashed box to eliminate it since it is the Case 1 with $q=k-2$.
\begin{figure}[htbp]
\centering
\includegraphics[height=5cm]{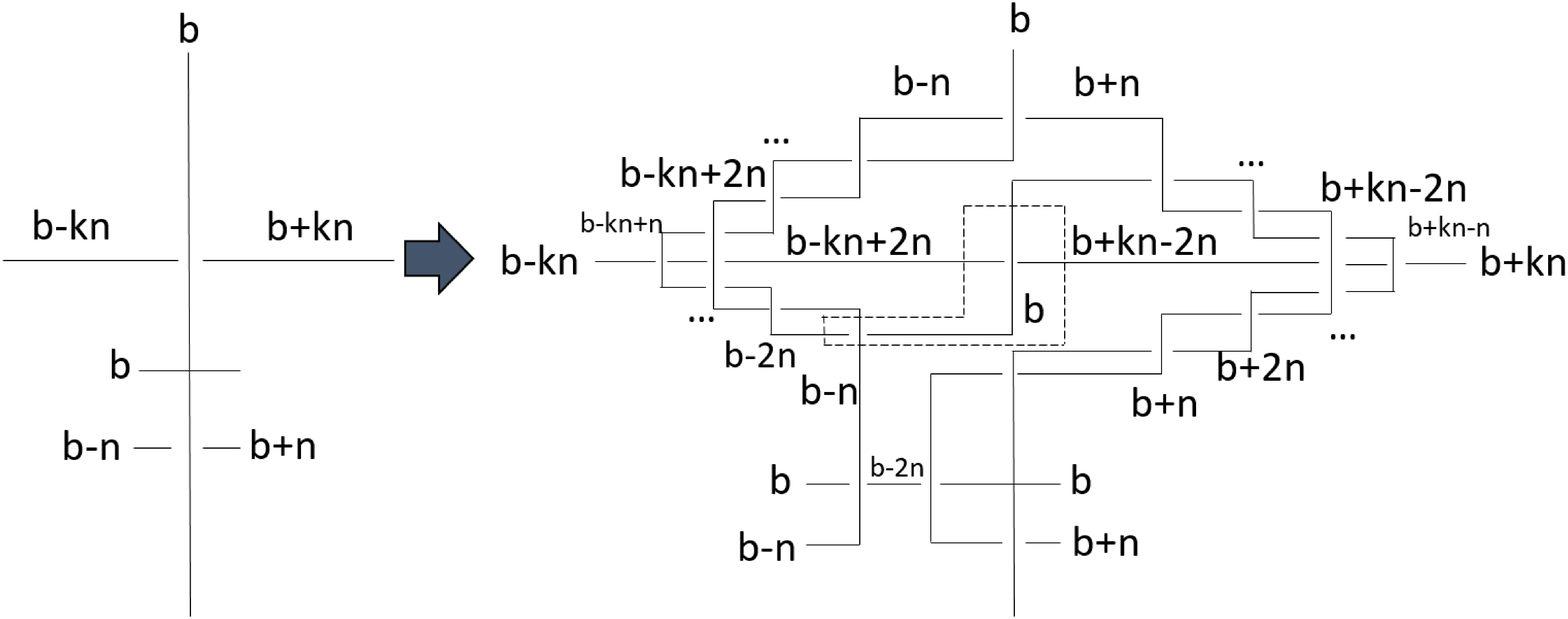}
\includegraphics[height=6.2cm]{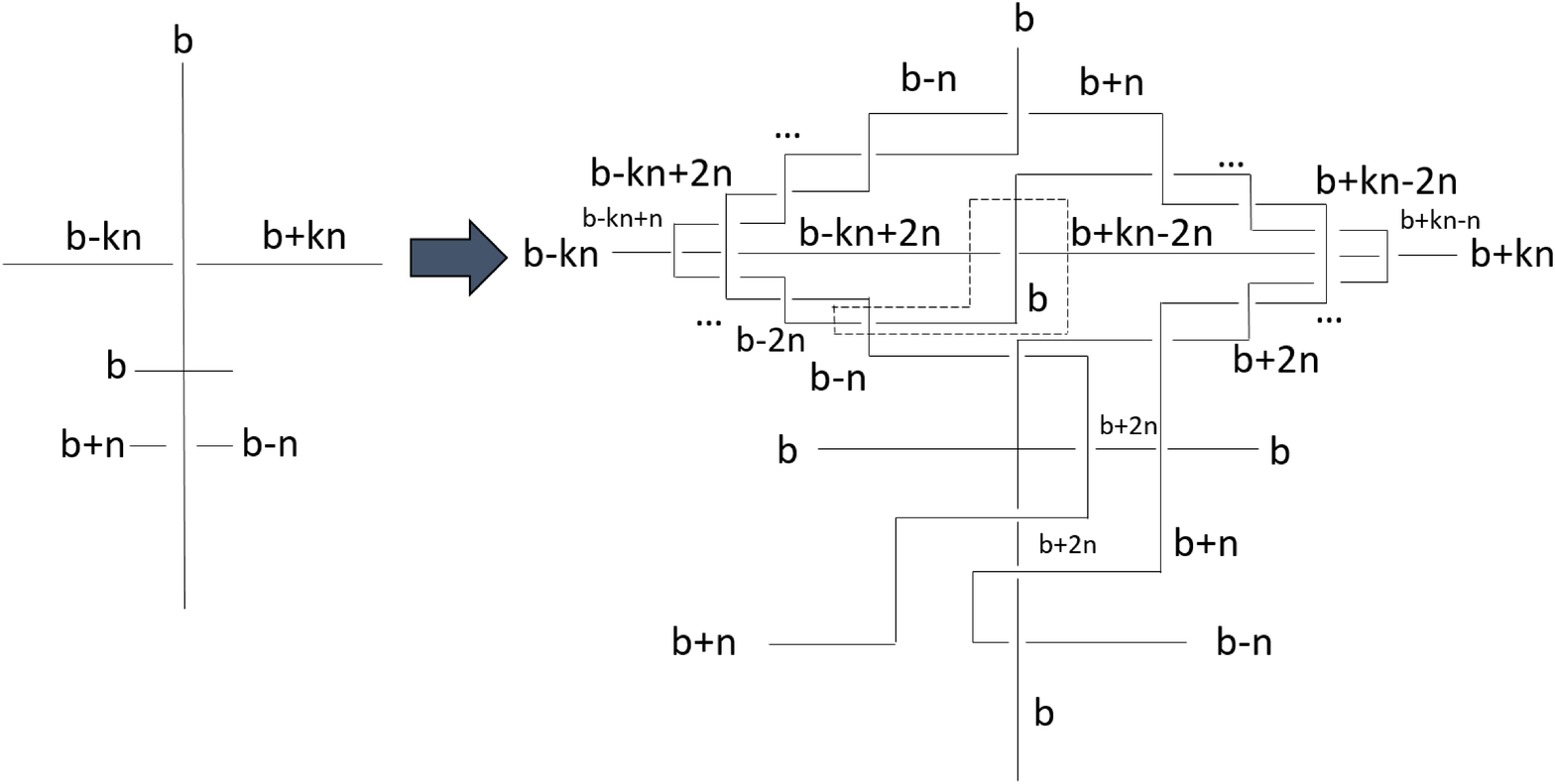}
\centerline{Figure 1.2.2}
\end{figure}

\noindent{\bf Case 3}: See Figure 1.1 (3). When $q=2$, we can eliminate $b|b+2n|b+4n$ and $b|b-2n|b-4n$ as shown in Figure 1.3.1, and any newly created crossing in the process of elimination is either a $0$-diff crossing or an $n$-diff crossing.
\begin{figure}[htbp]
\centering
\includegraphics[height=5.2cm]{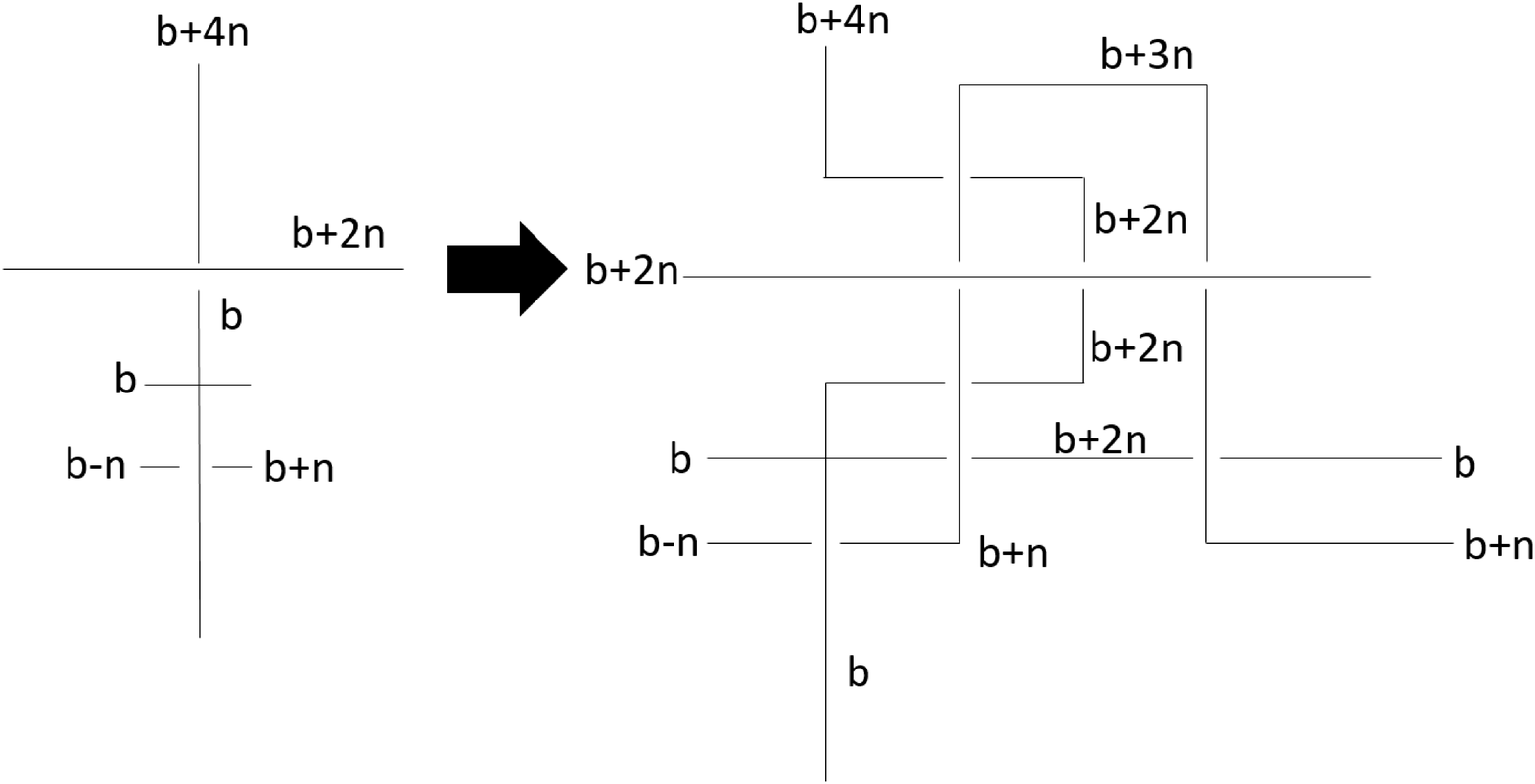}
\includegraphics[height=5.2cm]{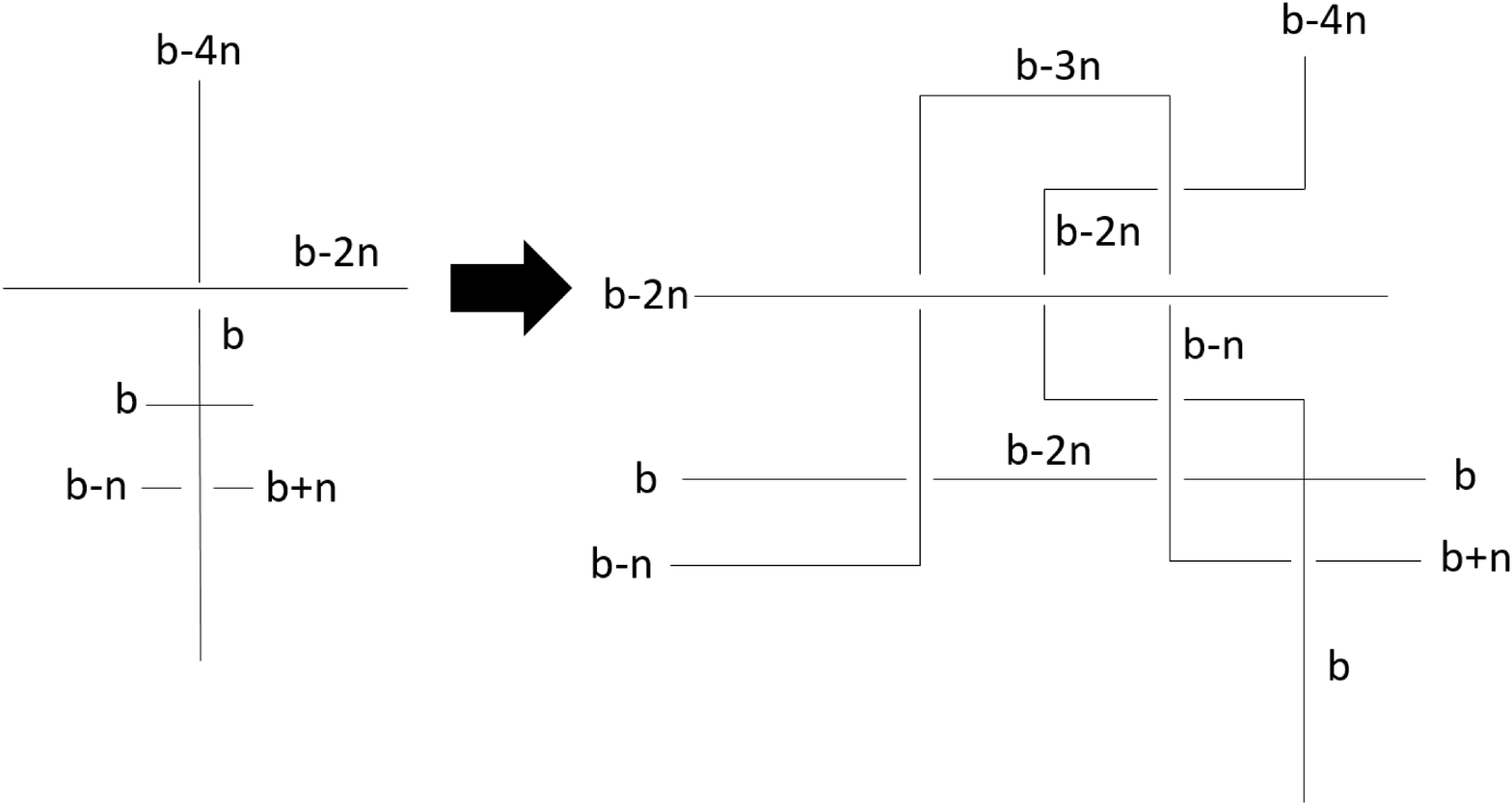}
\centerline{Figure 1.3.1}
\end{figure}
When $q=k$, $k\geq3$, we can operate as shown in Figure 1.3.2. Let $b'=b+kn$, $b''=b-kn$. Although there is newly created $b+n|b+kn|b+(2k-1)n$ crossing, that is $b'-(k-1)n|b'|b'+(k-1)n$ in the figure (above) and newly created $b-n|b-kn|b-(2k-1)n$, that is $b''+(k-1)n|b''|b''-(k-1)n$ in the figure (below), both the structures contained in the dashed boxes are the Case 2 with $q=k-1$. We can eliminate them as done in Case 2.
\begin{figure}[htbp]
\centering
\includegraphics[height=8cm]{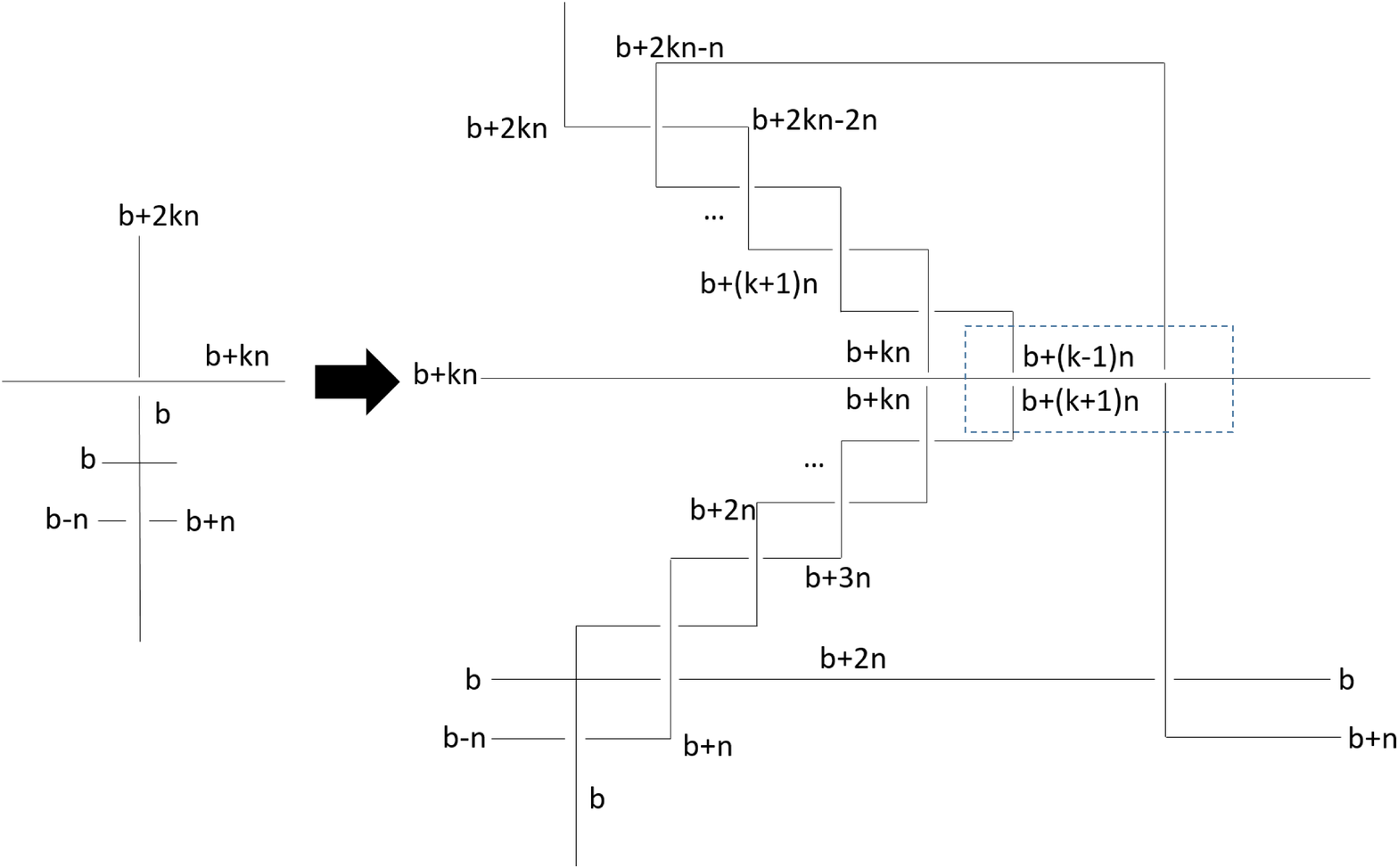}
\includegraphics[height=8cm]{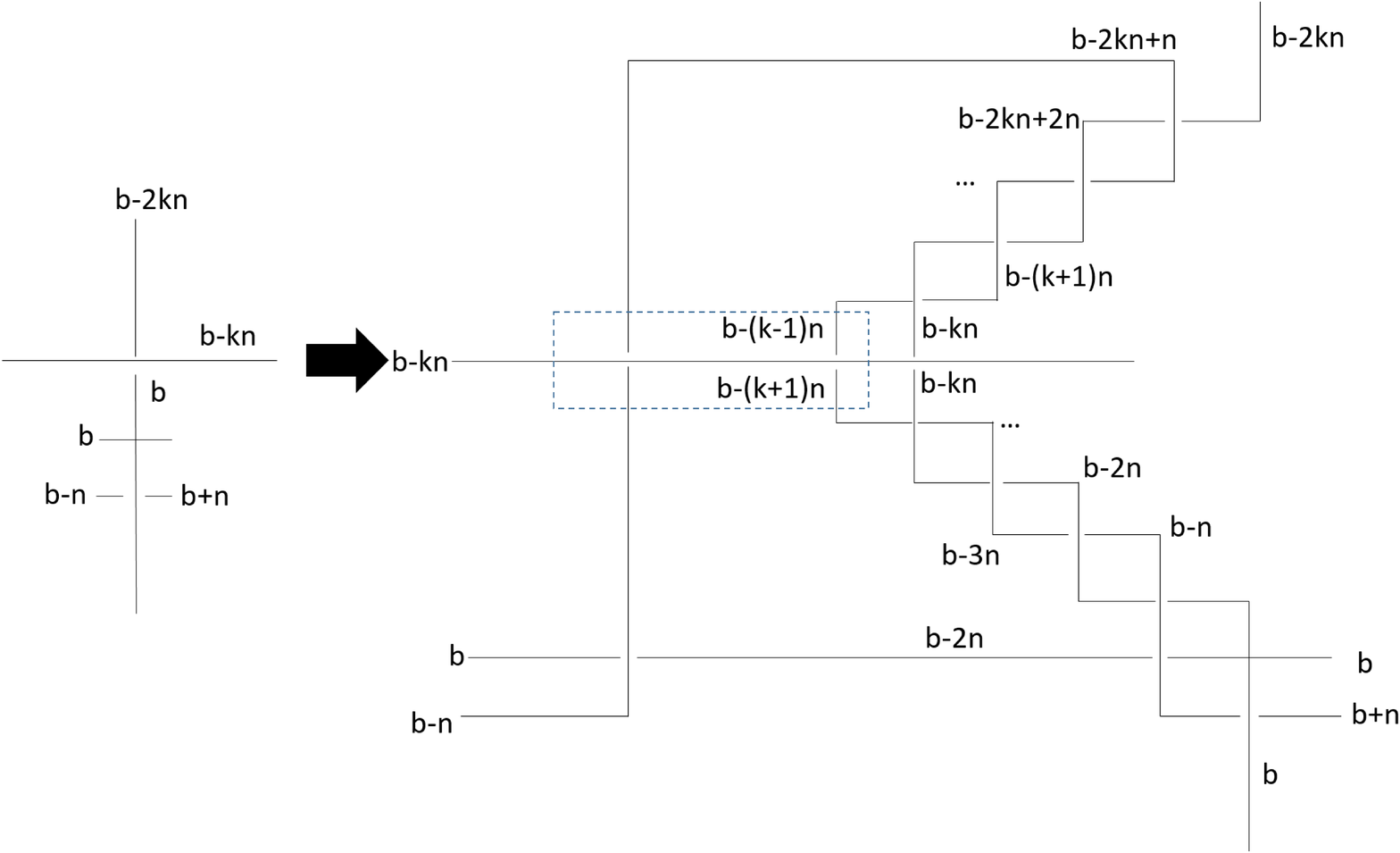}
\centerline{Figure 1.3.2}
\end{figure}

\noindent{\bf Case 4}: See Figure 1.1 (4). There are 4 subcases.

For the first and second subcases, we can eliminate $b|b+kn|b+2kn$ and $b|b-kn|b-2kn$ as shown in Figure 1.4.1. Let $b'=b+kn$, $b''=b-kn$. The newly created $b-n|b+kn|b+(2k+1)n$, that is $b'-(k+1)n|b'|b'+(k+1)n$ in the figure (above) and the newly created $b+n|b-kn|b-(2k+1)n$, that is $b''+(k+1)n|b''|b''-(k+1)n$ in the figure (below), contained in the dashed boxes are the Case 2 with $q=k+1$. We can eliminate them as done in Case 2, and any newly created crossing in the process of elimination is either a $0$-diff crossing or an $n$-diff crossing.
\begin{figure}[htbp]
\centering
\includegraphics[height=8cm]{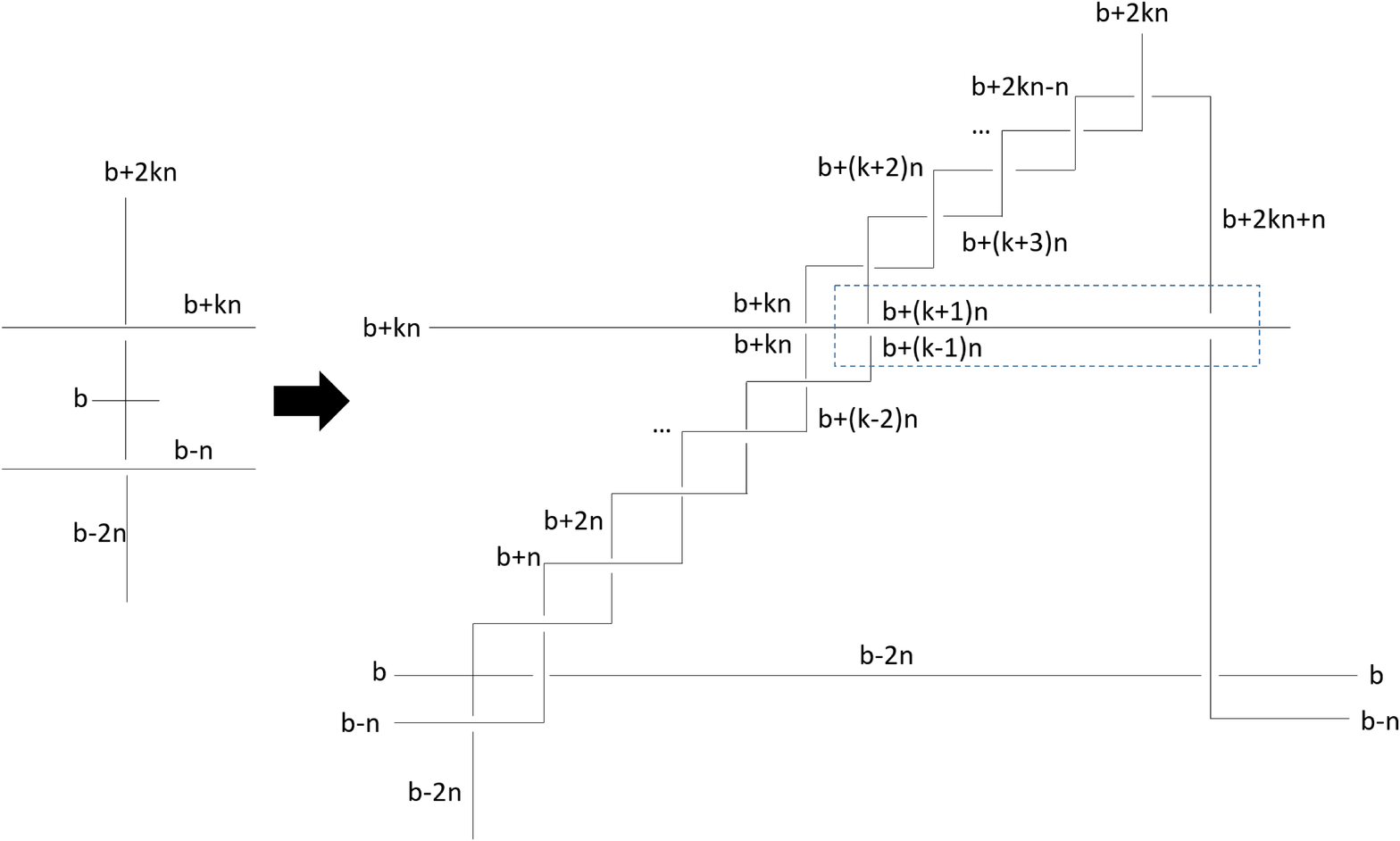}
\includegraphics[height=8cm]{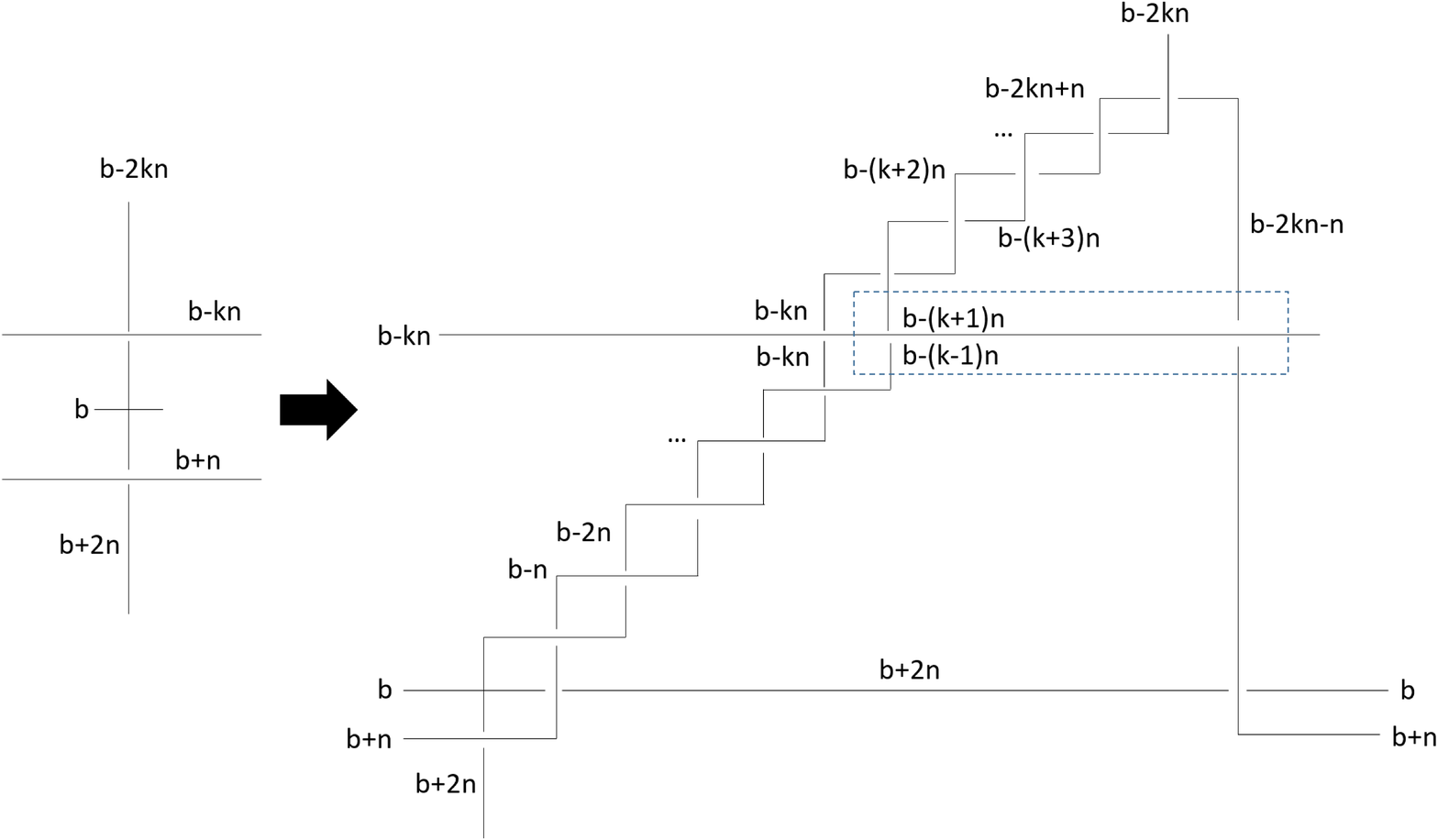}
\centerline{Figure 1.4.1}
\end{figure}

For the third and fourth subcases. When $q=2$, we can eliminate $b|b-2n|b-4n$ and $b|b+2n|b+4n$ as shown in Figure 1.4.2, and any newly created crossing in the process of elimination is either a $0$-diff crossing or an $n$-diff crossing. When $q=k$, $k\geq3$, we can eliminate $b|b-kn|b-2kn$ and $b|b+kn|b+2kn$ as shown in Figure 1.4.3. Let $b'=b-kn$, $b''=b+kn$. The newly created $b-n|b-kn|b-(2k-1)n$, that is $b'+(k-1)n|b'|b'-(k-1)n$ in the figure (above) and newly created $b+n|b+kn|b+(2k-1)n$, that is $b''-(k-1)n|b''|b''+(k-1)n$ in the figure (below), contained in the dashed boxes are the Case 2 with $q=k-1$. We can eliminate them as done in Case 2, and any newly created crossing in the process of elimination is either a $0$-diff crossing or an $n$-diff crossing.

\begin{figure}[htbp]
\centering
\includegraphics[height=5.2cm]{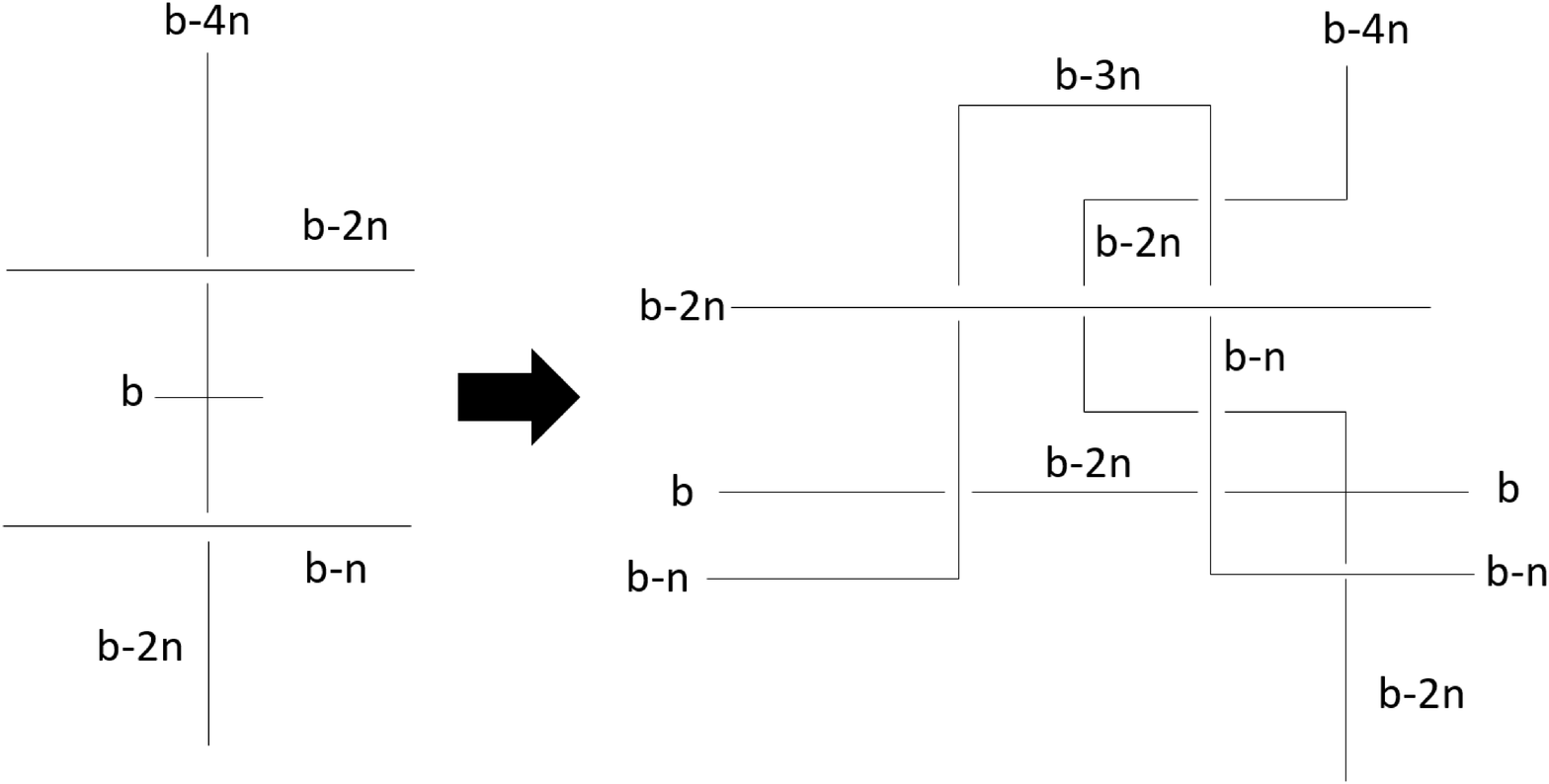}
\includegraphics[height=5.2cm]{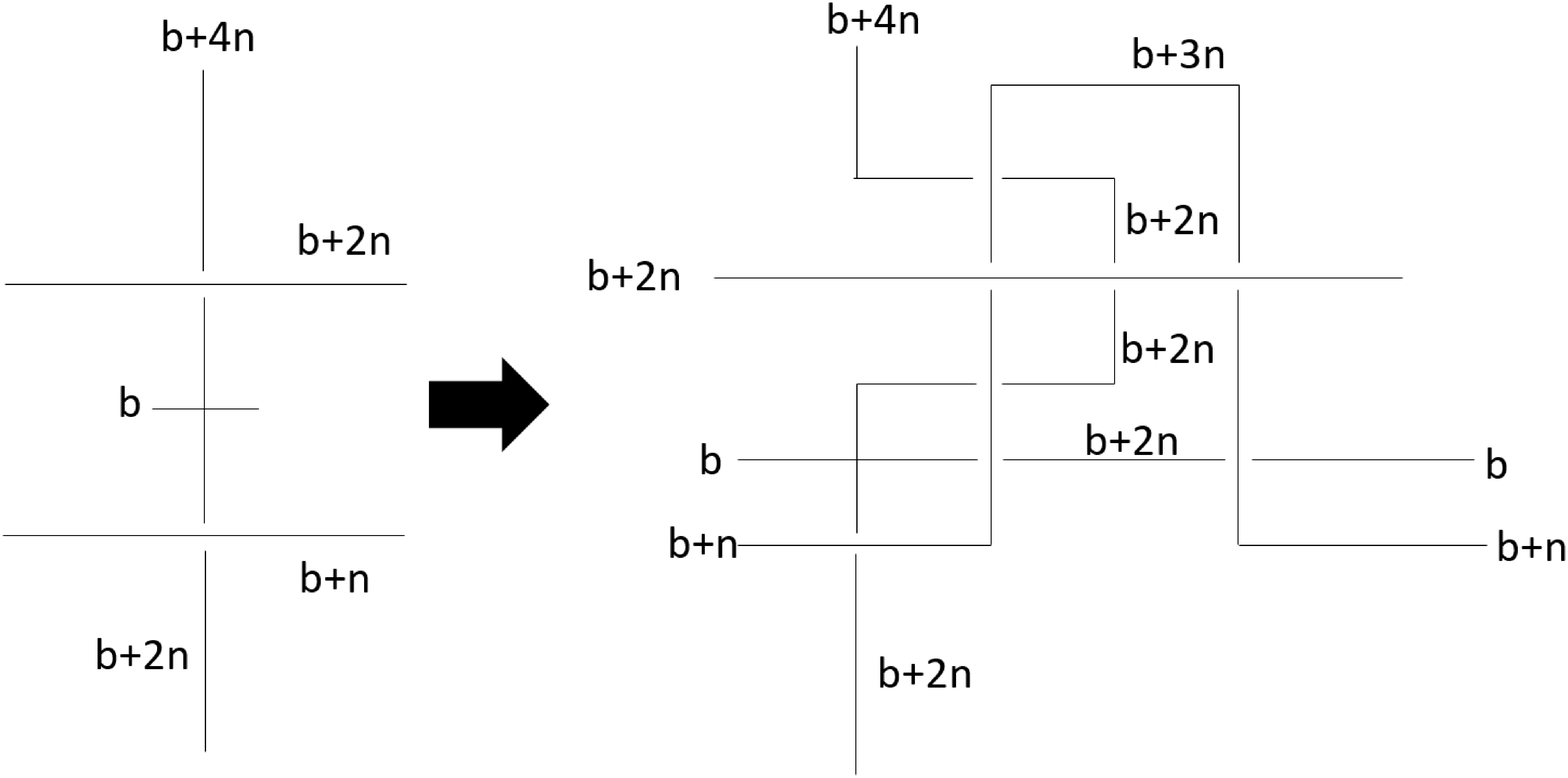}
\centerline{Figure 1.4.2}
\end{figure}

\begin{figure}[htbp]
\centering
\includegraphics[height=8cm]{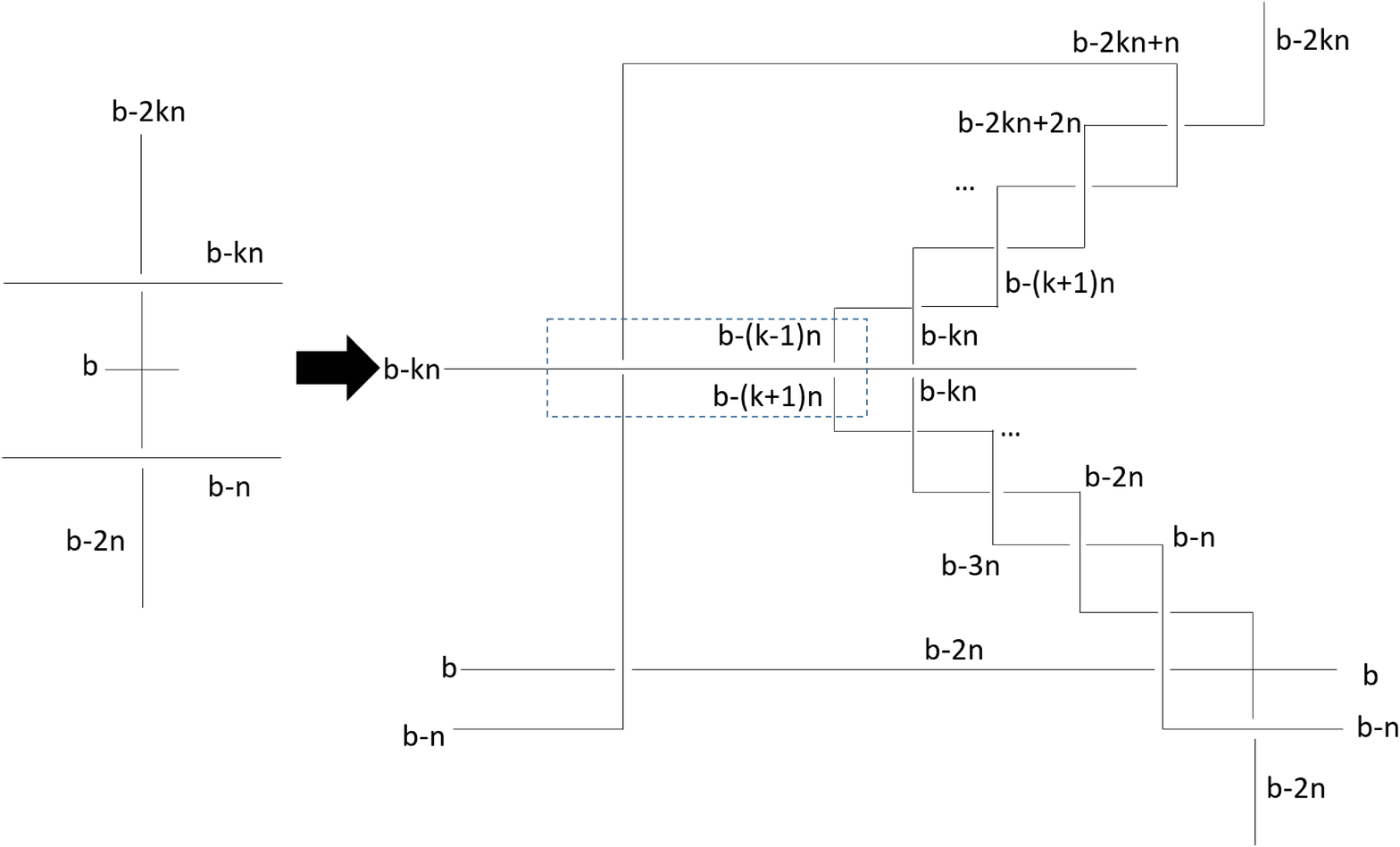}
\includegraphics[height=8cm]{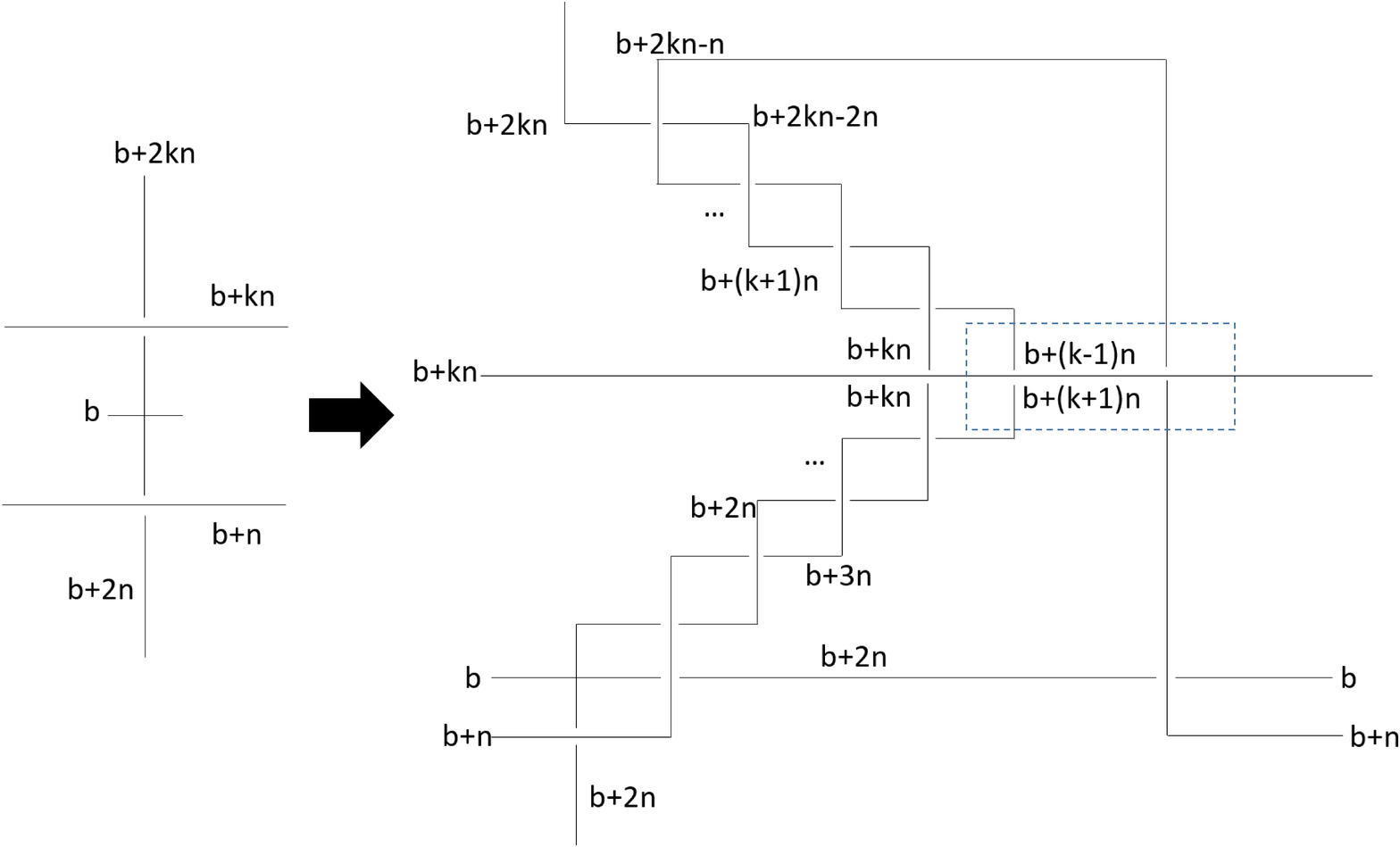}
\centerline{Figure 1.4.3}
\end{figure}

\end{proof}

\begin{Lemma}\label{22}
Let $L$ be a $ \mathbb{Z} $-colorable link, and $\gamma$ a $ \mathbb{Z} $-coloring on a diagram $D$ of $L$. If there exists a pair of adjacent $n$-diff crossing and $m$-diff crossing, we can convert this local structure to a new one by equivalent local moves, containing only $0$ or $d$-diff crossings, where $d=\gcd(m,n)$.
\end{Lemma}

\begin{proof}
Without loss of generality, we assume $m>n$. We shall prove Lemma \ref{22} by induction on $n$. If $n=1$, then $d=1$. Applying Lemma \ref{11}, we can convert the local structure and obtain a new local structure containing only $0$ or $1$-diff crossings, thus the lemma holds. Next, assuming that when $n\leq z-1$, $z\geq2$, the lemma holds, we shall prove when $n=z$, the lemma also holds.

If $m=qz$ ($q\geq1, q\in\mathbb{N}^{+}$), then $d=z$. Applying Lemma \ref{11}, we can obtain a new structure, containing only 0 or $d$-diff crossings, which means the lemma holds. If $m=qz+r$ ($q\geq1, 0<r<z$), we shall further prove, for each case, we can create a $r$-diff crossing which is adjacent to a $z$-diff crossing and the difference of any newly created crossing in the operation is $kd$ for some nonnegative integer $k$. Note that $n=z$,$m=qz+r$ and $d=\gcd(m,n)=\gcd(n,r)$.

For (1) in Figure 1.1 with $n=z$, $m=qz+r$, we can operate as shown in Figure 2.1.

For (2) in Figure 1.1 with $n=z$, $m=qz+r$, we can operate as shown in Figure 2.2.1, 2.2.2.

For (3) in Figure 1.1 with $n=z$, $m=qz+r$, we can operate as shown in Figure 2.3.

For (4) in Figure 1.1 with $n=z$, $m=qz+r$, we can operate as shown in Figure 2.4.1, 2.4.2.

In the figures above, there always exists a $r$-diff crossing which is adjacent to a $z$-diff crossing. See dashed boxes. After proving that, we can turn the adjacent $z$-diff crossing and $r$-diff crossing to a new local structure, containing only $0$ or $d$-diff crossings (and a $d$-differ crossing is always created) by hypothesis induction, because $0<r<z$. Moreover, because the difference of any newly created crossing in the operation of creating a $r$-diff crossing is $kd$ for some nonnegative integer $k$, the great common divisor of the $d$-diff crossing and its adjacent crossings is also $d$. Applying Lemma \ref{11}, we can turn these kinds of adjacent crossings to new structures containing only $0$ or $d$-diff crossings, and continuing this process repeatedly, the lemma will hold.
\end{proof}

\begin{figure}[htbp]
\centering
\includegraphics[height=8cm]{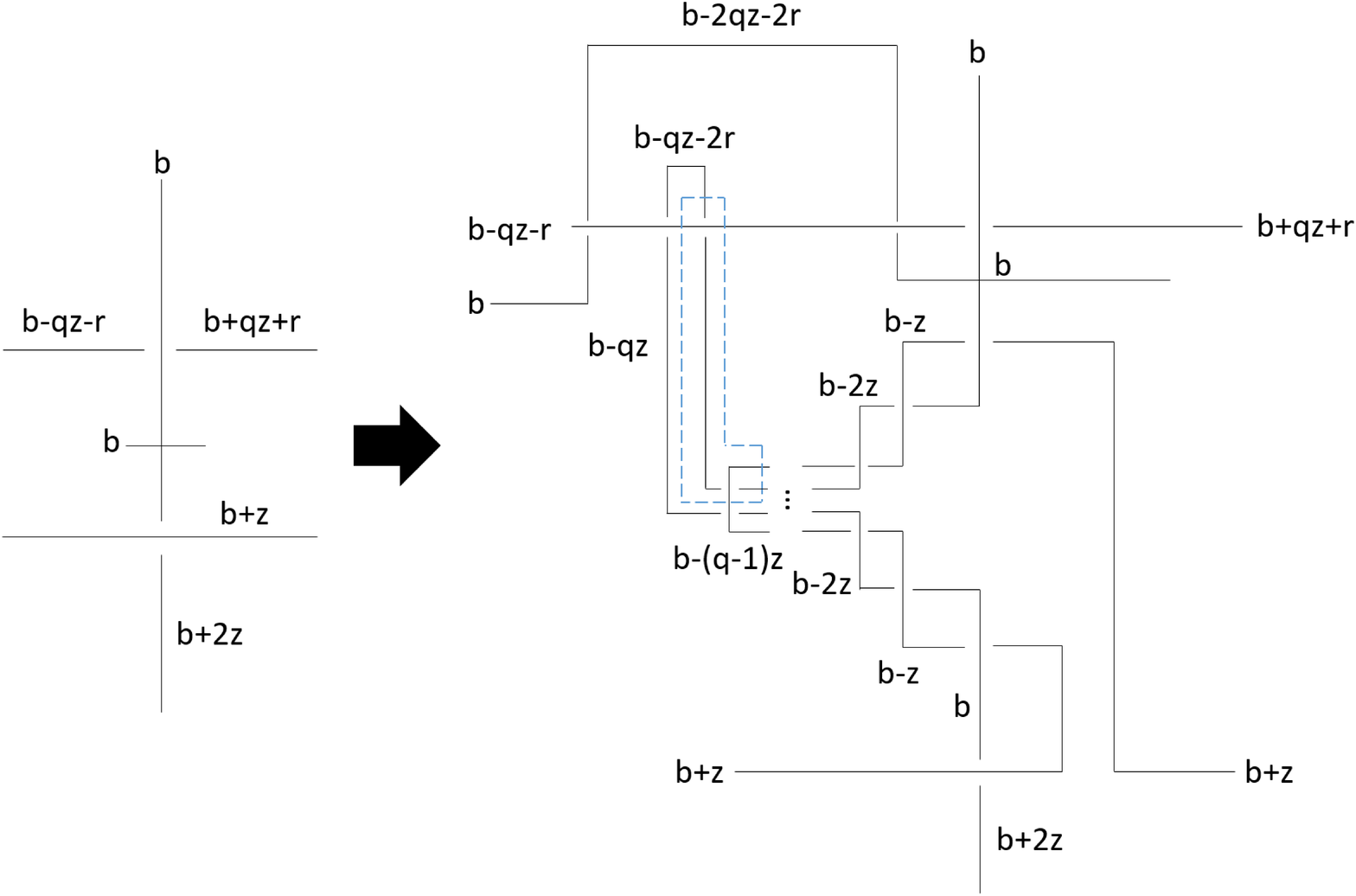}
\includegraphics[height=8cm]{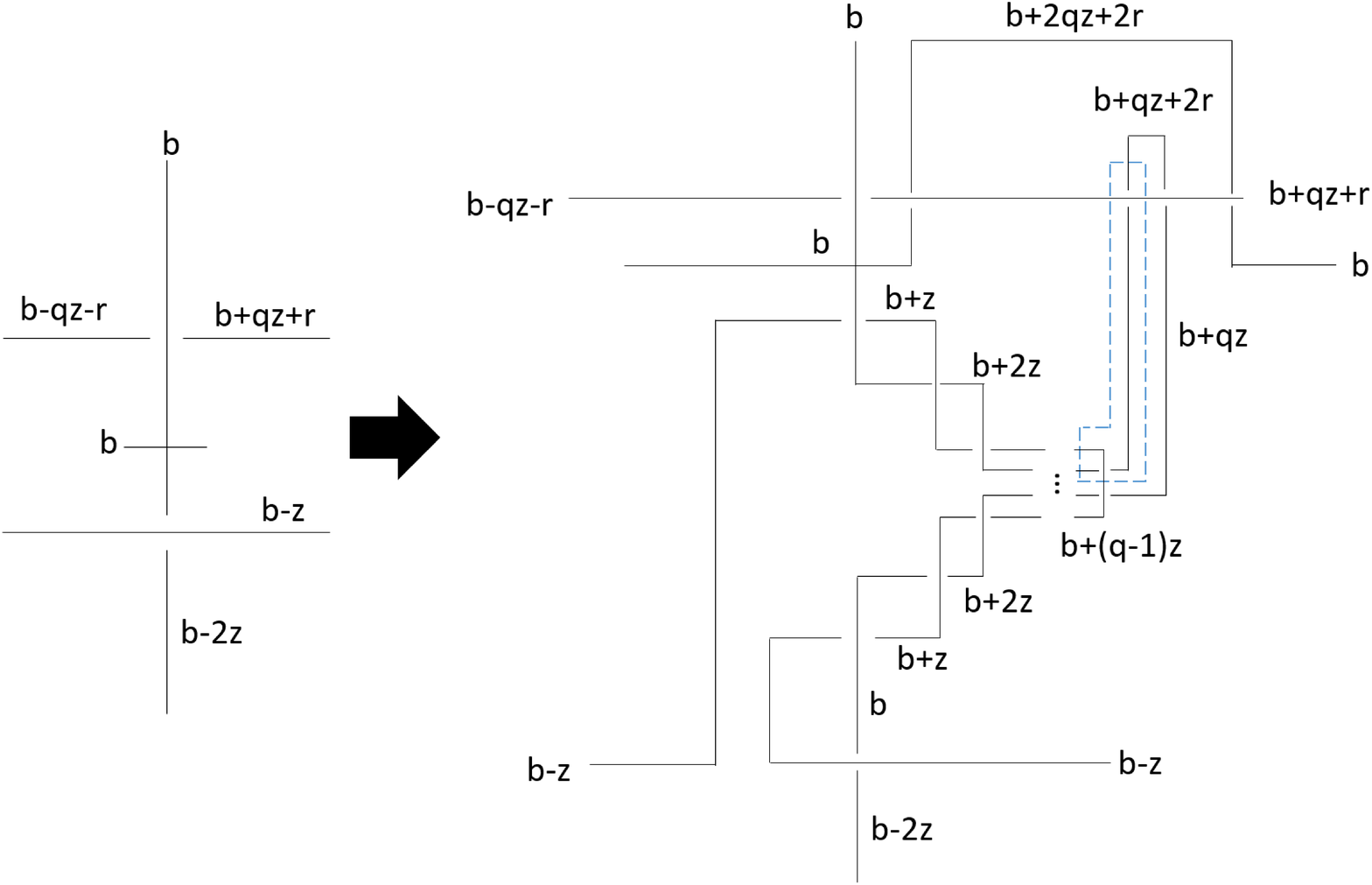}
\centerline{Figure 2.1}
\end{figure}

\begin{figure}[htbp]
\centering
\includegraphics[height=8cm]{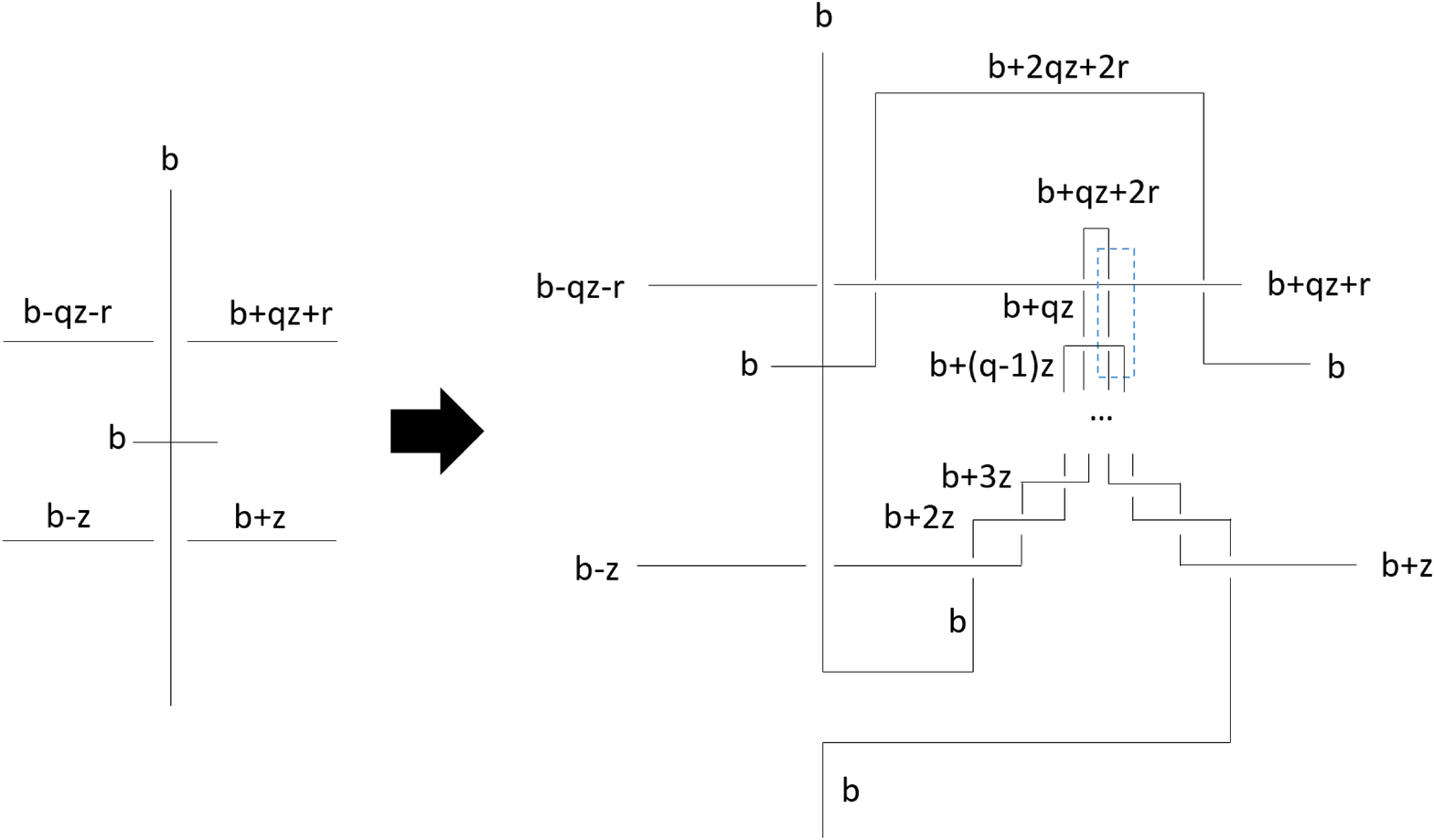}
\centerline{Figure 2.2.1}
\includegraphics[height=8cm]{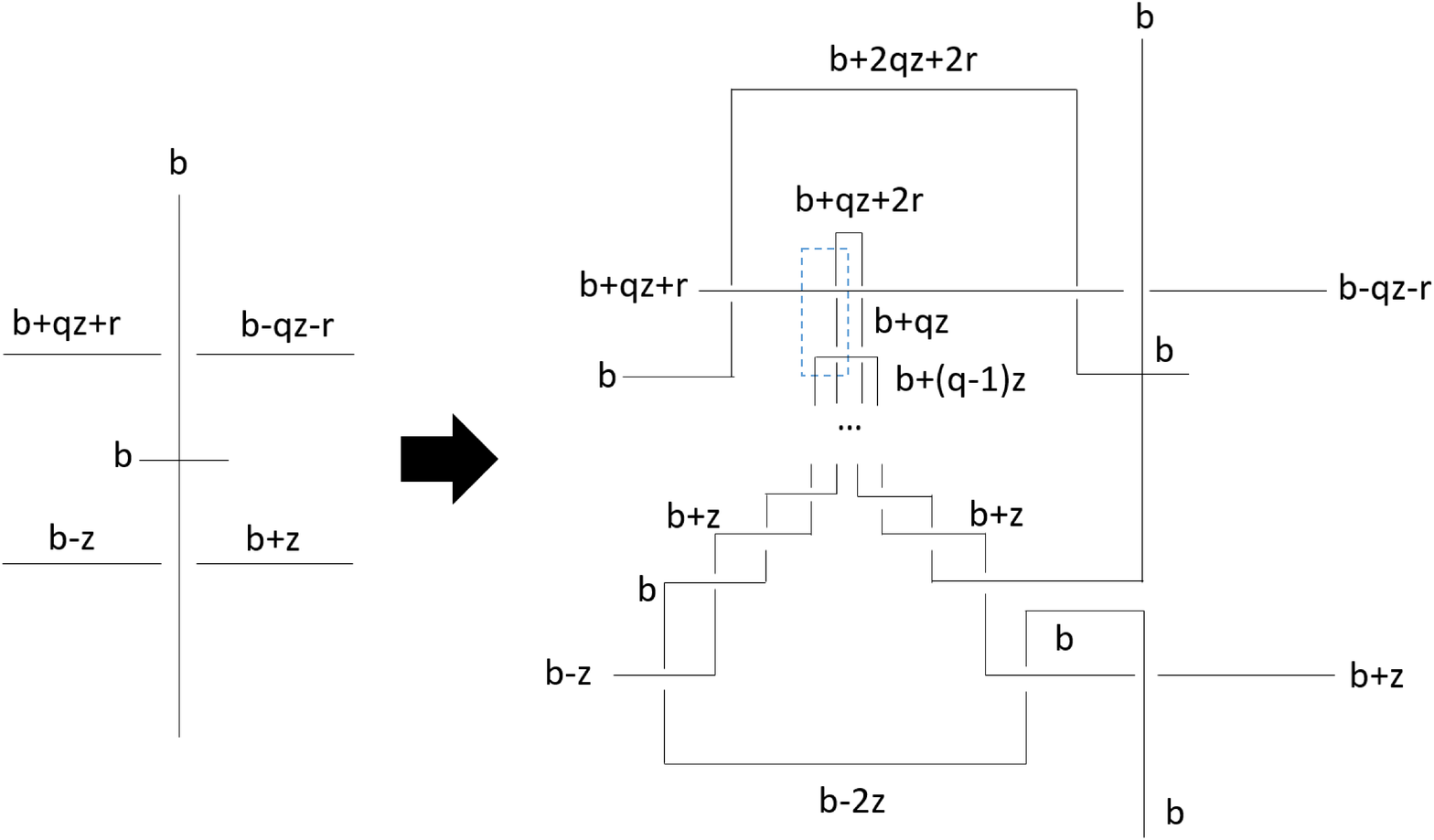}
\centerline{Figure 2.2.2}
\end{figure}

\begin{figure}[htbp]
\centering
\includegraphics[height=8cm]{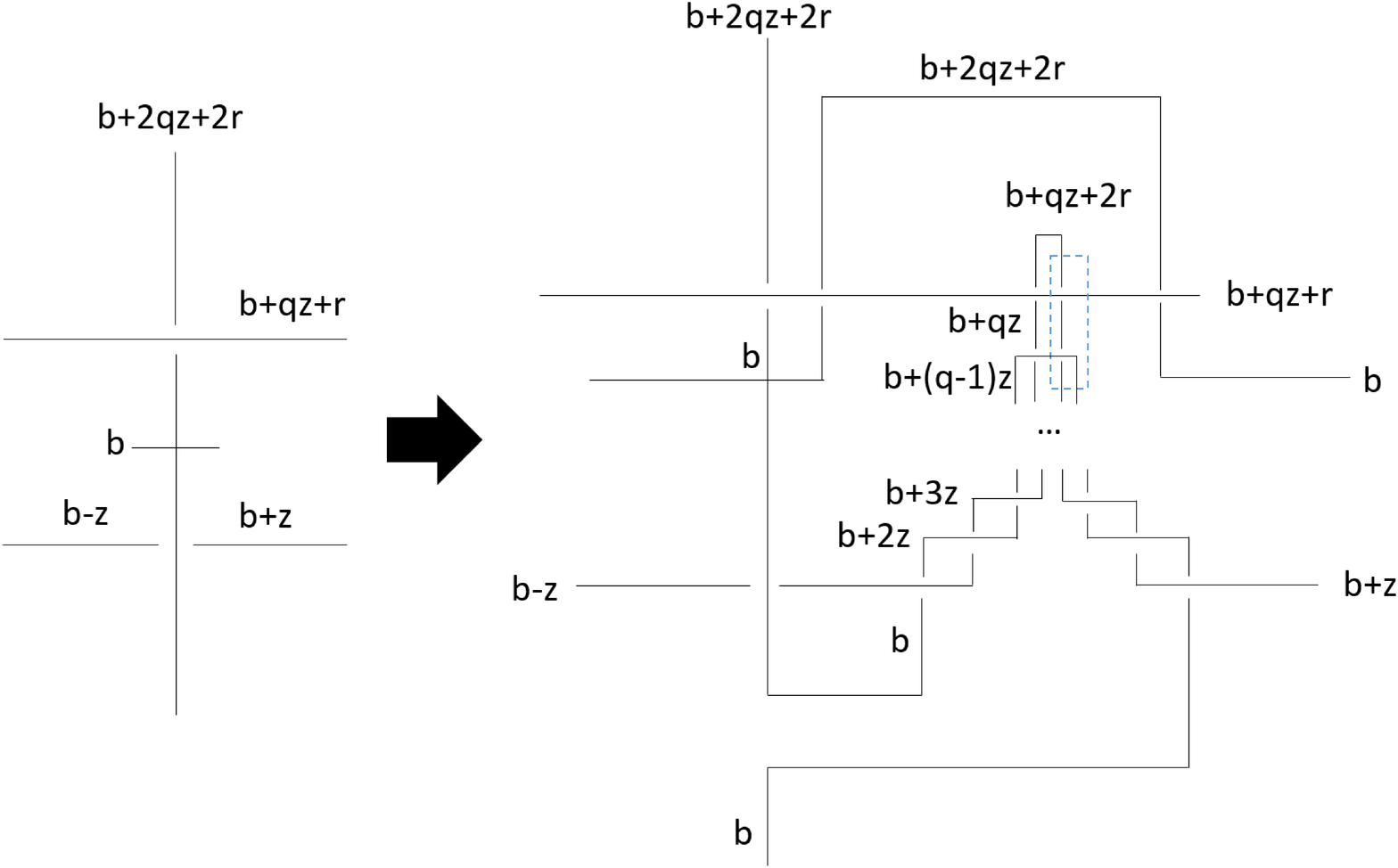}
\includegraphics[height=8cm]{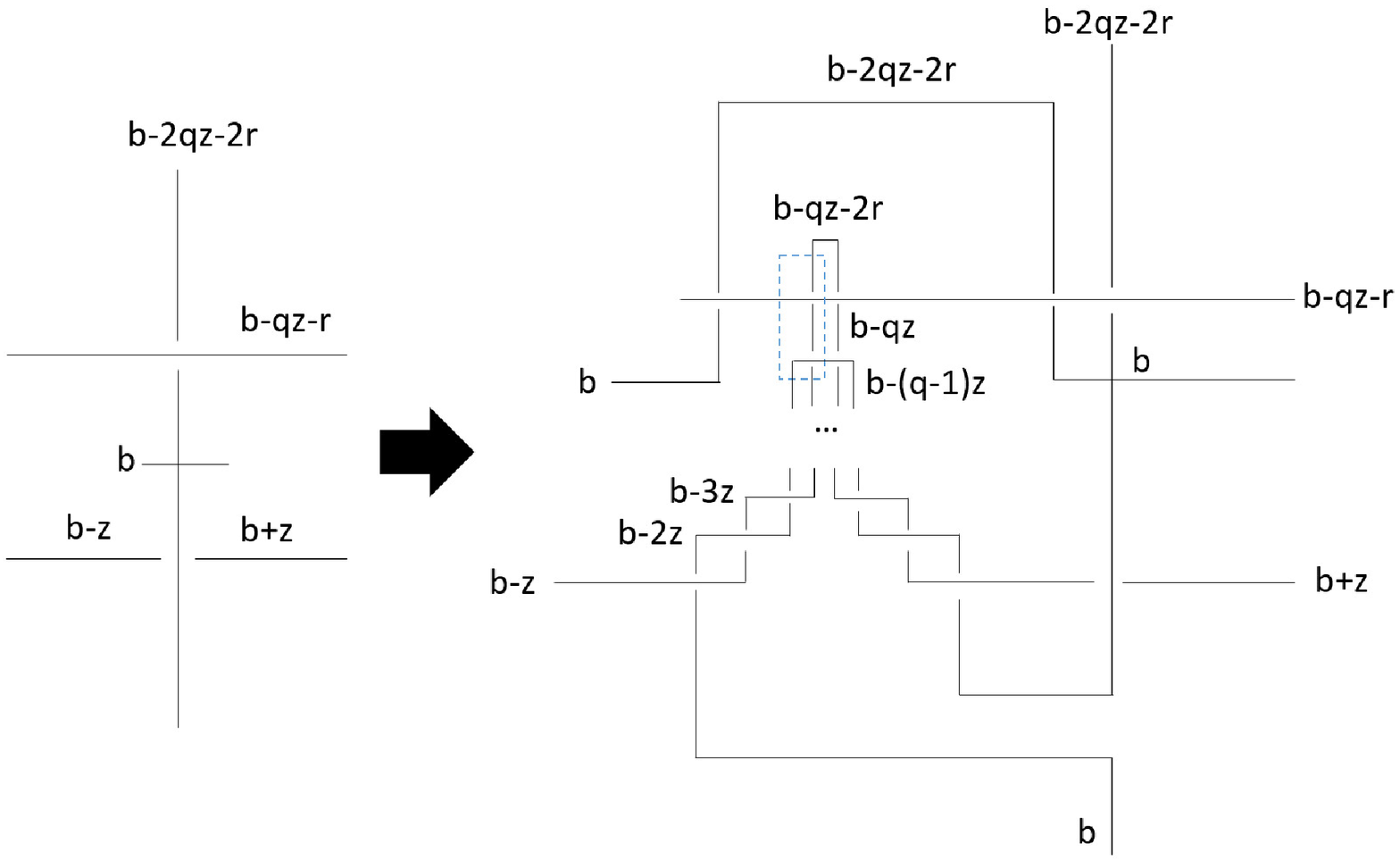}
\centerline{Figure 2.3}
\end{figure}

\begin{figure}[htbp]
\centering
\includegraphics[height=8cm]{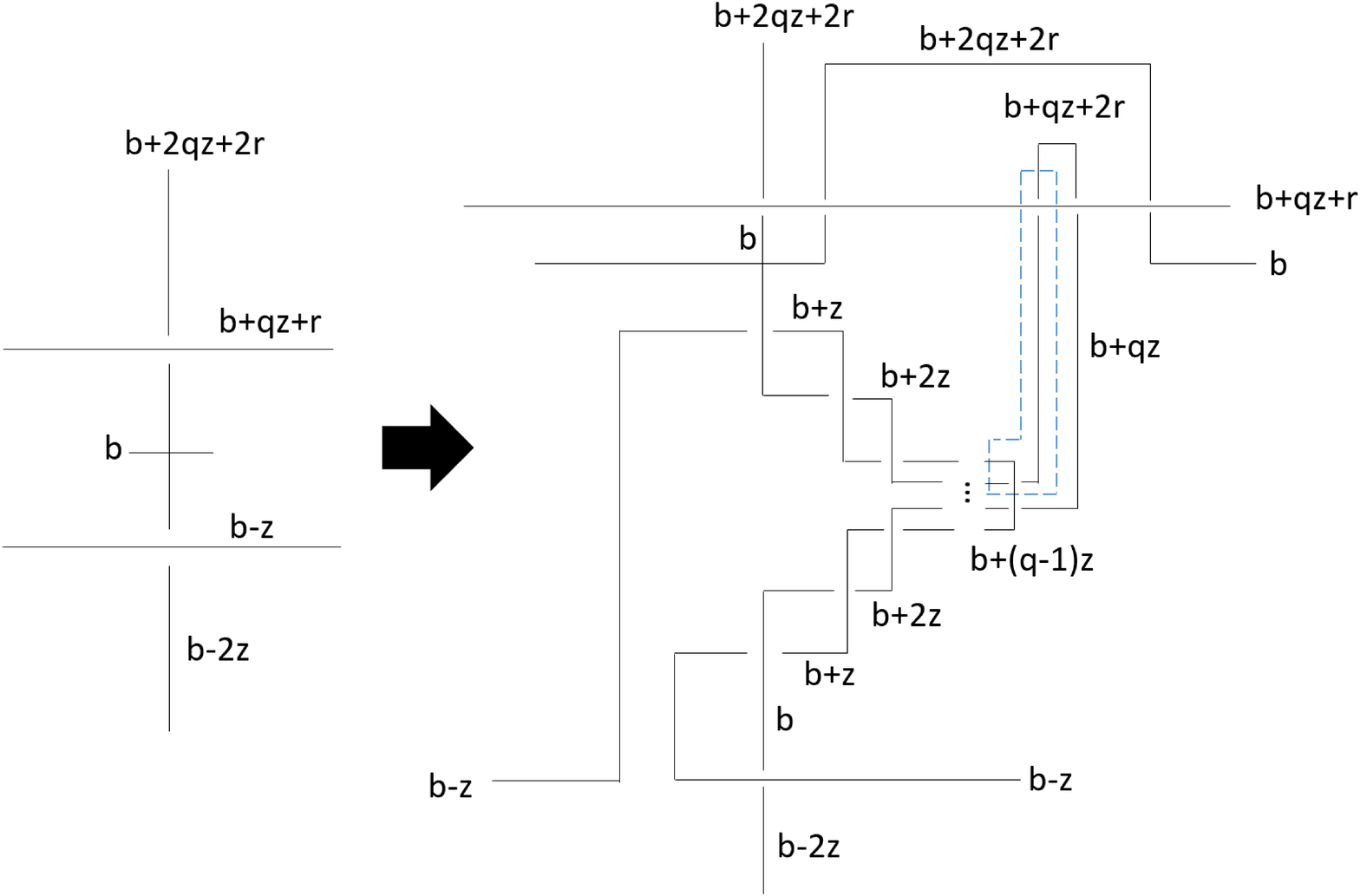}
\includegraphics[height=8cm]{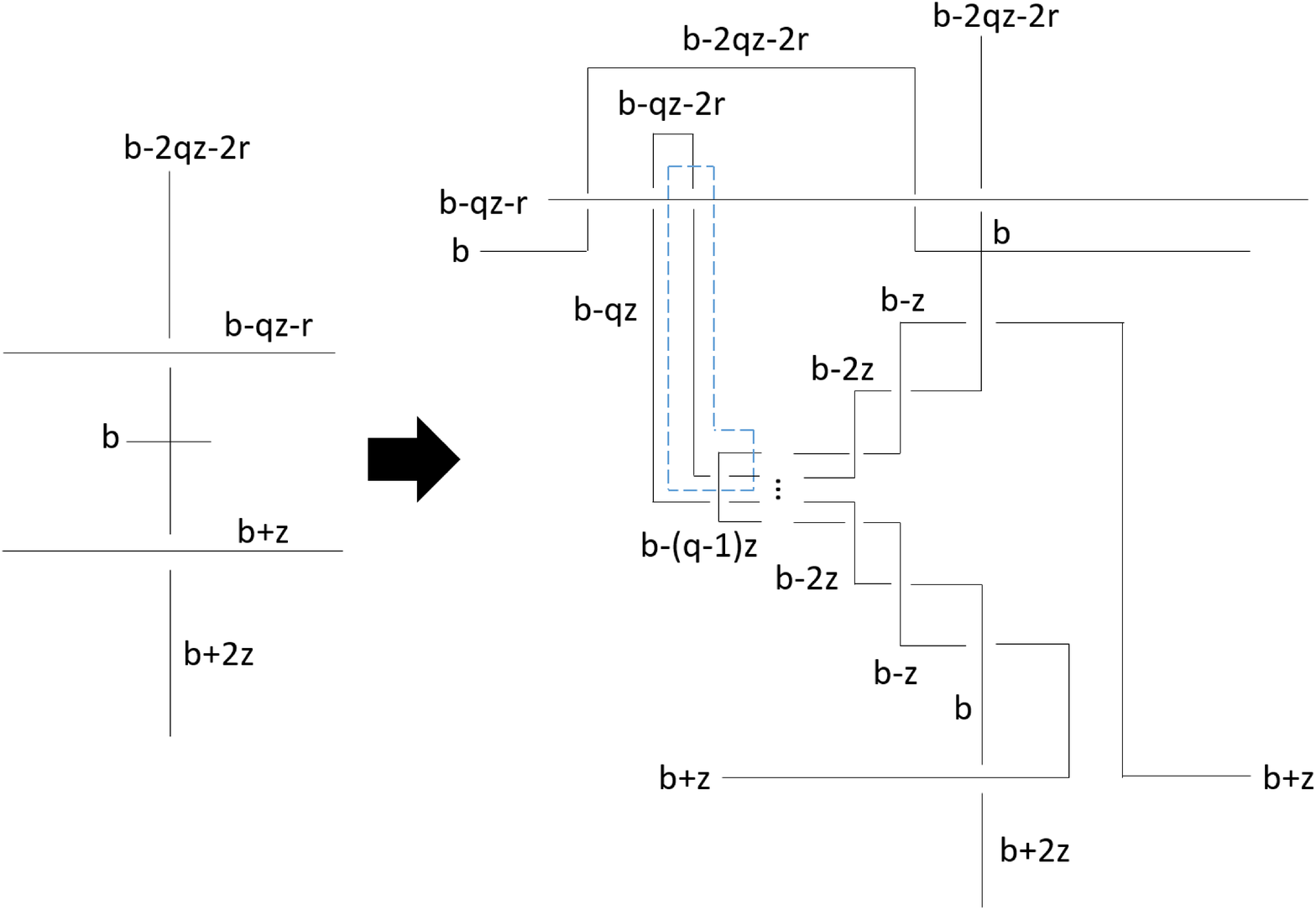}
\centerline{Figure 2.4.1}
\end{figure}

\begin{figure}[htbp]
\centering
\includegraphics[height=8cm]{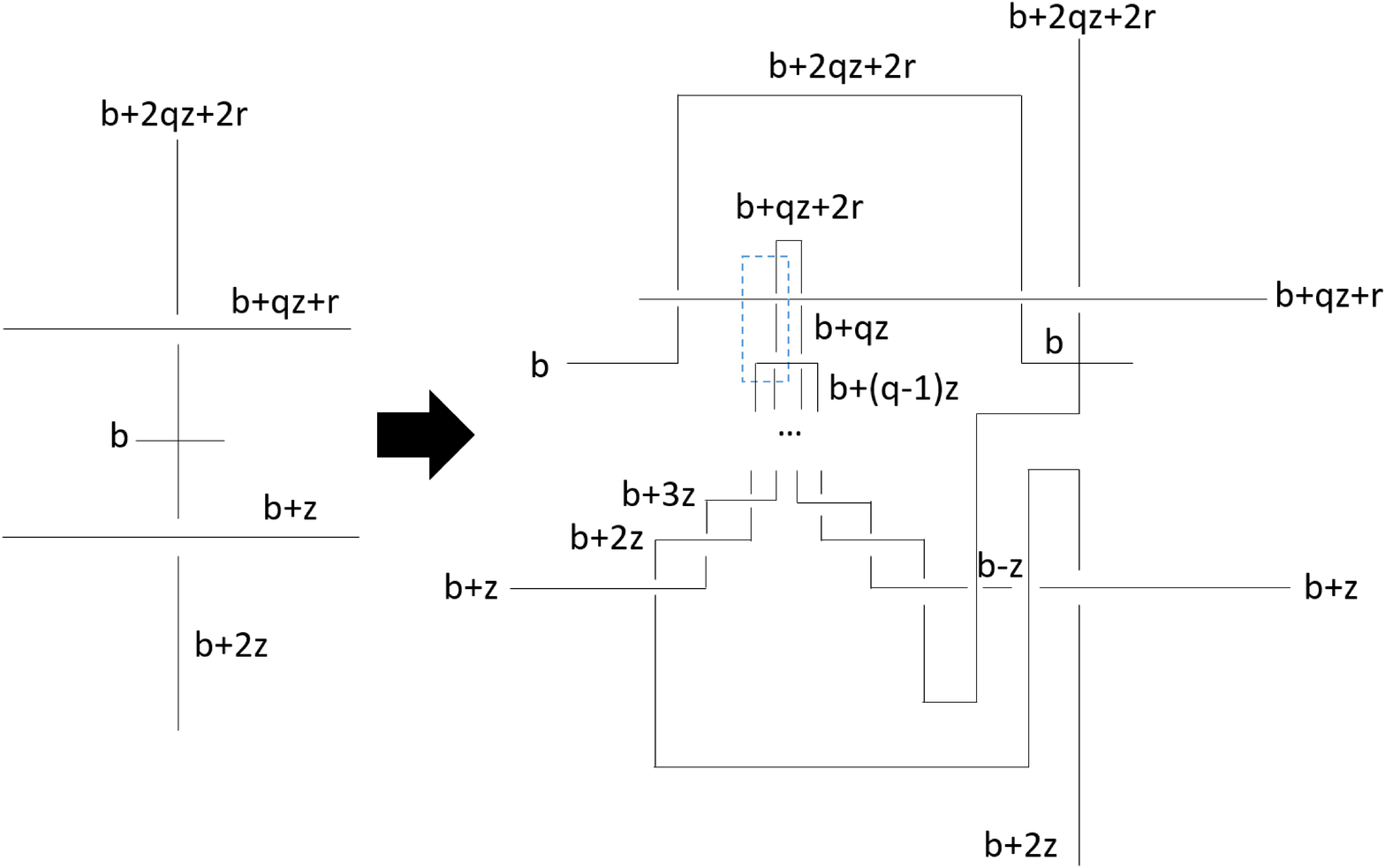}
\includegraphics[height=8cm]{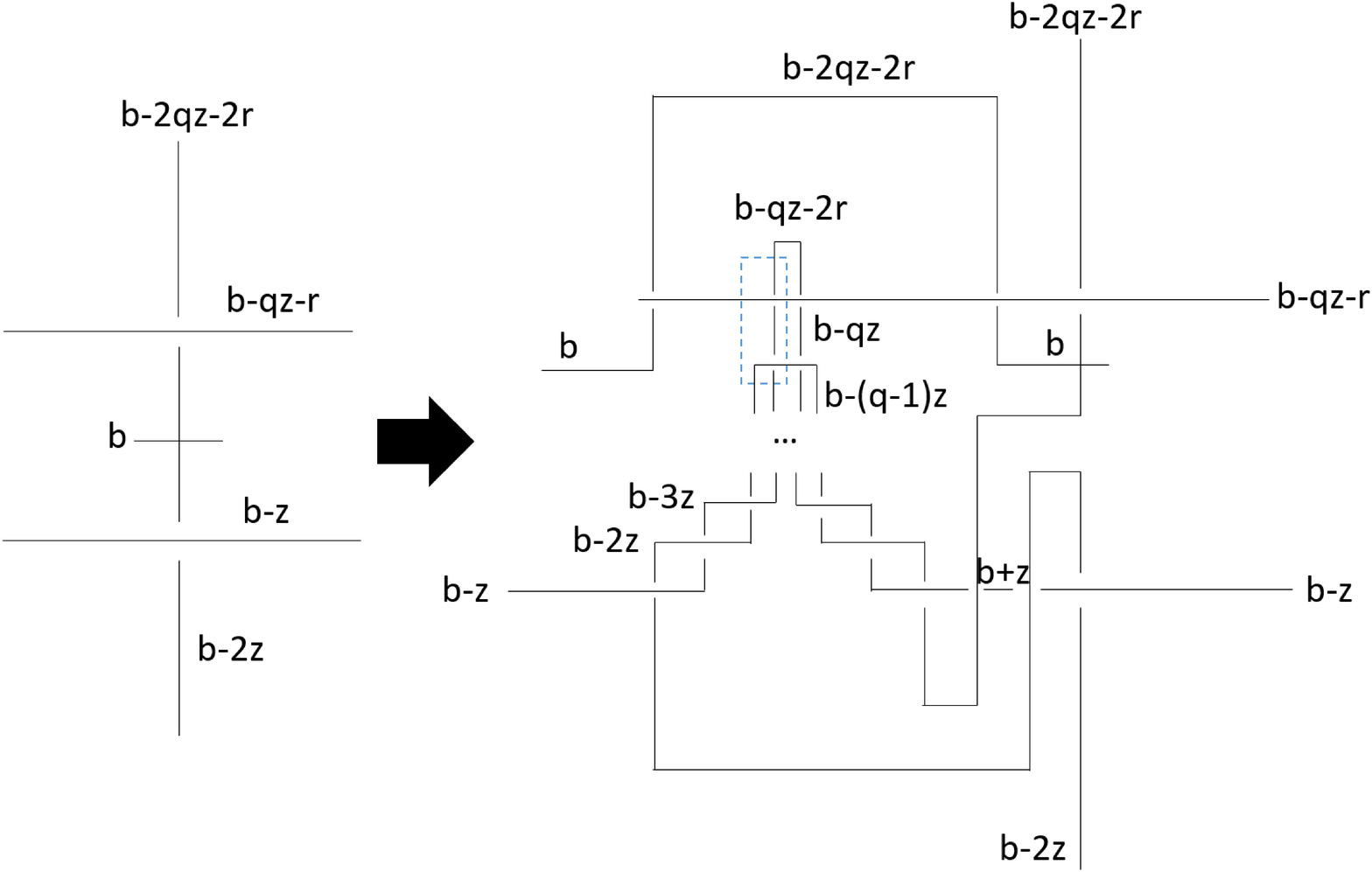}
\centerline{Figure 2.4.2}
\end{figure}

\noindent \textbf{Proof of Theorem 2.2.}
Let $L$ be a non-splittable $\mathbb{Z} $-colorable link, and $\gamma$ a $ \mathbb{Z} $-coloring on a diagram $D$ of $L$. If this coloring is not simple, then there exists a pair of adjacent $m$-diff and $n$-diff crossings, where $m,n\in \mathbb{N}^{+}$ and $m\neq n$. Applying the Lemma \ref{22} to this pair of crossings, we can obtain a new equivalent diagram by converting the local structure containing $m$-diff crossing, $n$-diff crossing and $0$-diff crossings between them to a new local structure containing only $0$ or $d$-diff crossings, where $d=\gcd(m,n)$. Since $L$ is non-splittable, continue the above process repeatedly, we can obtain a simple coloring.

\section*{Acknowledgements}
\noindent

This work is supported by NSFC (No. 11671336) and President's Funds of Xiamen University (No. 20720160011).

\end{document}